\documentclass[10pt,reqno,letterpaper]{amsart}
\usepackage{amssymb}
\usepackage{graphicx,color}
\usepackage{hyperref}
\usepackage[normalem]{ulem}
\usepackage{fullpage} % Smaller margins (1in). With the option [cm] margins are 1.5cm

\usepackage{scalerel}[2014/03/10]
\usepackage[usestackEOL]{stackengine}

\def\dashint{\,\ThisStyle{\ensurestackMath{%
  \stackinset{c}{.2\LMpt}{c}{.5\LMpt}{\SavedStyle-}{\SavedStyle\phantom{\int}}}%
  \setbox0=\hbox{$\SavedStyle\int\,$}\kern-\wd0}\int}
\def\ddashint{\,\ThisStyle{\ensurestackMath{%
  \stackinset{c}{.2\LMpt}{c}{.5\LMpt+.2\LMex}{\SavedStyle-}{%
    \stackinset{c}{.2\LMpt}{c}{.5\LMpt-.2\LMex}{\SavedStyle-}{%
      \SavedStyle\phantom{\int}}}}\setbox0=\hbox{$\SavedStyle\int\,$}\kern-\wd0}\int}

\theoremstyle{plain}

\begin{document}

%- Theorems and similar stuff: ------

\theoremstyle{plain}
\newtheorem{theorem}{Theorem} [section]
\newtheorem{corollary}[theorem]{Corollary}
\newtheorem{lemma}[theorem]{Lemma}
\newtheorem{proposition}[theorem]{Proposition}
\newtheorem{example}[theorem]{Example}

%- Definitions -----------------------------

\theoremstyle{definition}
\newtheorem{definition}[theorem]{Definition}
\theoremstyle{remark}
\newtheorem{remark}[theorem]{Remark}

%- Numeracion ------------------------------------
\numberwithin{theorem}{section}
\numberwithin{equation}{section}
%\numberwithin{figure}{section}

\title{A Nonlinear Mean Value Property for Monge-Amp\`ere}

\thanks{P.B. partially supported by the Academy of Finland project no. 298641. \\
\indent F.C. partially supported by MINECO grants MTM2016-80474-P, MTM2017-84214-C2-1-P, and MTM2017-85934-C3-2-P (Spain),
ICMAT Severo Ochoa project SEV-2015-0554, and 
ERC grants 307179-GFTIPFD and 834728-QUAMAP.
 \\
\indent J.D.R. partially supported by 
CONICET grant PIP GI No 11220150100036CO
(Argentina) and UBACyT grant 20020160100155BA (Argentina).
}

\author[P. Blanc]{Pablo Blanc}
\address{Department of Mathematics and Statistics, University of Jyv\"askyl\"a, PO Box 35, FI-40014 Jyv\"askyl\"a, Finland}
\email{pblanc@dm.uba.ar}

\author[F. Charro]{Fernando Charro}
\address{Department of Mathematics, Wayne State University, 656 W. Kirby, Detroit, MI 48202, USA}
\email{fcharro@wayne.edu}
\thanks{}

\author[J. J. Manfredi]{Juan J. Manfredi}
\address{
Department of Mathematics,
University of Pittsburgh, Pittsburgh, PA 15260, USA.
}
\email{manfredi@pitt.edu}

\author[J. D. Rossi]{Julio D. Rossi}
\address{Departamento  de Matem\'atica, FCEyN, Universidad de Buenos Aires,
 Pabell\'on I, Ciudad Universitaria (1428),
Buenos Aires, Argentina.}
\email{jrossi@dm.uba.ar}

\date{}

\begin{abstract}

In recent years there has been an increasing interest in whether a mean value property, known to characterize harmonic functions, 
can be extended in some
weak form to solutions  of nonlinear equations. This question has been partially motivated by the surprising connection between Random Tug-of-War games and the normalized $p-$Laplacian discovered some years ago, where a nonlinear asymptotic 
mean value property for solutions of  a PDE is related to a dynamic programming principle for an appropriate  game. 
 Currently, asymptotic nonlinear mean value formulas are rare in the literature and our goal is to show that an asymptotic 
 nonlinear mean value formula holds for the classical Monge-Amp\`ere equation.
\end{abstract}

\maketitle

\begin{center}
\textit{Dedicated with admiration and affection to our friend Umberto Mosco on his $80^{th}$-birthday.} 
\end{center}

\section{Introduction}\label{section.intro}

It is well-known that there is a mean value formula that characterizes harmonic functions;  a function $u$ is harmonic 
($u$ is a solution to $\Delta u=0$) in a domain
$\Omega\subset\mathbb{R}^n$ 
 if and only if $u$ satisfies the mean value property
\begin{equation}\label{MVPrperty_HarmonicFunctions}
u(x)
=
\dashint_{B_\varepsilon(x)}u(y)\, dy
\end{equation}
for each $x\in\Omega$ and all $0 < \varepsilon< \textrm{dist}(x,\partial\Omega)$. 
In fact, a weaker statement  
known as asymptotic mean value property is enough to characterize harmonicity: A continuous function $u$ is harmonic in $\Omega$ if and only if 
\begin{equation}\label{Asymptotic.MVP.Harmonic}
u(x)=\dashint_{B_\varepsilon(x)} u(y)\,dy+o(\varepsilon^2)\quad\textrm{as}\ \varepsilon\to0,
\end{equation}
see \cite{Blaschke,Privaloff}.
The same results hold if we replace the average over the ball $B_\varepsilon(0)$ by an average over the sphere $\partial B_\varepsilon(0)$. 
Moreover, the mean value property can be extended to characterize sub- and superharmonic functions replacing the equality by the appropriate inequality
in \eqref{MVPrperty_HarmonicFunctions} and \eqref{Asymptotic.MVP.Harmonic}.
Furthermore, a discrete version also holds; a continuous function $u$ is harmonic  if and only if  
\begin{equation}\label{Asymptotic.MVP.Harmonic.discrete}
u(x)= \frac1n \sum_{j=1}^n \left\{ \frac12 u(x+\varepsilon e_j) + \frac12 u(x-\varepsilon e_j) \right\}
+o(\varepsilon^2)\quad\textrm{as}\ \varepsilon\to 0,
\end{equation}
where $\{e_1, \dots, e_n\}$ is the canonical basis of $\mathbb{R}^n$.

Extensions of these ideas to the classical Poisson equation $-\Delta u=f$  derive  from the formula
\begin{equation}\label{Laplaciano.formula.asintotica}
\Delta u(x)=\lim_{\varepsilon \to0}\left[\frac{2(n+2)}{\varepsilon^2}\left(\dashint_{B_\varepsilon(x)} u(y)\,dy-u(x)\right)\right],
\end{equation}
from where we conclude that when $f$ is continuous,   a smooth function  $u$  satisfies $-\Delta u=f$ in $\Omega$ if and only if
\begin{equation}
\label{Asymptotic.MVP.Harmonic.f}
u(x)=\dashint_{B_\varepsilon (x)} u(y)\,dy+\frac{\varepsilon^2}{2(n+2)}\,f(x)+o(\varepsilon^2)\qquad\textrm{as} \ \varepsilon \to0
\end{equation}
for each $x \in \Omega$.

Mean value formulas for operators other than the Laplacian are relatively rare in the literature. In \cite{Littman.et.al1963}, the authors prove a mean value theorem for linear divergence form operators with bounded measurable coefficients.
In this case, the statement of the mean value formula involves the Green function for the operator. In  \cite{Caf98} (see also \cite{BH2015,Caffarelli.Roquejoffre.2007}), a simpler statement in terms of mean value sets
 was pointed out (a set $D$ is a mean value set for the point $x\in D$ and the operator $L$ if the mean value property \eqref{MVPrperty_HarmonicFunctions} holds with $D$ replacing $B_\varepsilon(x)$ for every $u$ such that $Lu = 0$).
 It is worth mentioning that  mean value sets and their properties are yet to be fully understood, see \cite{Armstrong.2018,Armstrong.Blank.2019,Aryal.Blank.2019,Benson.Blank.LeCrone.2019,Ku}. For mean value properties in the sub-Riemannian 
 setting we refer to \cite{BoLan}.

In recent years there has been an increasing interest 
in whether the mean value property can be extended to  nonlinear operators
such as the normalized infinity Laplacian
\[
\Delta_\infty^Nu=\frac{1}{|\nabla u|^2}\sum_{i,j=1}^n u_{x_ix_j}u_{x_i}u_{x_j}
\]
and the normalized (also called $1-$homogeneous) $p$-Laplacian, 
\[
\Delta_p^N u=
|\nabla u|^{2-p}\,\text{div}\big(|\nabla u|^{p-2}\nabla u\big)
=
\Delta u+ (p-2)\,\Delta_{\infty}^{N}u,
\]
for $1<p<\infty$.
 This question was originally motivated
   by the surprising connection 
   between 
   dynamic programming principles and the normalized infinity Laplacian discovered in \cite{LeGruyerArcher}, \cite{LeGruyer} and its connection with Random Tug-of-Wag games in  \cite{PSSW}. A 
nonlinear mean value property for $p$-harmonic functions first appeared in  \cite{[Manfredi et al. 2010]}  motivated by the Random Tug-of-War games with noise in \cite{PSSW2}. 
It was proved in \cite{[Manfredi et al. 2010]}  that $p$-harmonic functions
are characterized by the fact that they satisfy the following nonlinear asymptotic mean value formula
in the viscosity sense
 \begin{equation}\label{MVFormula}
u(x)
=
\left(\frac{p-2}{p+n}\right)
\left(\frac{\max_{\overline{B}_\varepsilon (x)}u+\min_{\overline{B}_\varepsilon  (x)}u}{2}\right)
+
\left(\frac{2+n}{p+n}\right)\dashint_{B_\varepsilon  (x)}u(y)\, dy+o(\varepsilon^2)\quad\textnormal{as}\ \ \varepsilon \to
0.
\end{equation}
See also also \cite{Angel.Arroyo.Tesis} and the recently published book \cite{[Blanc and Rossi 2019]} for historical references and  more general equations. 
\par
Observe that  the mean value formula  holds in a generalized or viscosity sense,
meaning that whenever a smooth test function touches $u$ from above (respectively below) at  a point $x$, the mean value formula \eqref{MVFormula} 
is satisfied with $\leq$ (resp. $\geq$) for the test function at  $x$.
In addition, in certain cases  \eqref{MVFormula}  holds point-wise. This is the case when the dimension $n=2$ and $1<p<\infty$ (\cite{LindqvistManfredi, ArroyoLLorente}). However, it does not hold for $p=\infty$ for the
Aronsson function, the $\infty-$harmonic function $u(x,y) = x^{4/3} - y^{4/3}$, see \cite{[Manfredi et al. 2010]}. 
Whether we have a point-wise characterization in higher dimensions remains an  open problem to the best of our knowledge.
The nonlinear mean value property \eqref{MVFormula} has a game-theoretic interpretation for which it is essential that the coefficients add up to 1, so that they play the role of conditional probabilities. 
Note that only the first term in \eqref{MVFormula} is nonlinear, with more weight the larger the $p$, while the second one is linear. In particular, the case $p=2$ reduces to
\eqref{Asymptotic.MVP.Harmonic}, the asymptotic mean value property that characterizes harmonic functions.
For mean value properties for the $p$-Laplacian in the Heisenberg group see \cite{LMan} and for the standard variational $p-$Laplacian \cite{dTLin}.

Our goal in this paper is to investigate the validity of mean value properties for solutions to  the Monge-Amp\`ere equation, 
$$\det D^2u=f,$$ which  arises in many areas of analysis, geometry, and applied mathematics. Some important examples are  the prescribed Gaussian curvature problem or Minkowski problem, optimal mass transportation and the construction of antennas and reflectors, see \cite{Guido.Alessio,Figalli.2017, Glimm.Oliker,  Glimm.Oliker2,Gutierrez,Villani.topics}.

In the  Dirichlet problem for the Monge-Amp\`ere equation, one prescribes a smooth domain $\Omega\subset\mathbb{R}^n$, boundary data $g(x)$ on $\partial \Omega$, and a right-hand side $f(x)$ in $\Omega$, and studies  existence and regularity of a function $u$ that verifies
\begin{equation}\label{problem.intro.model}
\left\{
\begin{array}{ll}
 \det D^2 u=f \quad & \textrm{in}\ \Omega,\\
 u=g\quad & \textrm{on}\ \partial \Omega.
\end{array}
\right. 
\end{equation}
For the problem to fit into the framework of the theory of fully nonlinear elliptic equations, one must look for 
convex solutions $u$ to ensure that $\det D^2 u$ is
indeed a monotone function of $D^2u$. Thus, we require the right hand side $f(x)$ to be non-negative. We also assume that the domain $\Omega$ is
convex. The convexity of $\Omega$ is required in order to construct appropriate smooth subsolutions that act as lower barriers, see \cite{Ca-Ni-Sp,Ca-Ni-Sp2}.
\par
The Monge-Amp\`ere equation can be expressed as an infimum of a family of linear operators as follows,
\begin{equation}\label{first.main.ingredient.intro.88}
 n\big(\det D^2u(x)\big)^{1/n}=\inf_{\det A=1} {\rm trace}(A^tD^2u(x)A)
\end{equation}
(see Remark \ref{remark.polar.decomposition} and Lemma \ref{caract.determ} below).
Notice that the condition $\det(A)=1$ allows the matrix $A$ to have arbitrarily small eigenvalues, which shows the degeneracy of the operator; this is immediately apparent in dimension $n=2$, where the two eigenvalues of the matrix must be reciprocal in order to keep the determinant equal to 1, but one of them
can be very small if the other is very large.

As mentioned above, our goal is to establish a mean value property for a convex function $u$ that satisfies
$\det D^2u=f$.
In fact, the convexity of $u$ automatically implies that the value of $u$ at a given point cannot be above its average over a ball, that is,
\begin{equation}\label{MVF.intro.convexity}
u(x)
\leq
\dashint_{B_\varepsilon(x)}u(y)\, dy.
\end{equation}
One could also see this formula as a consequence of the subharmonicity of $u$. 
A refinement comes from the Arithmetic-Geometric mean inequality, which implies
\[
\Delta u\geq n(\det D^2u)^{1/n}=n f^{1/n}.
\]
Then, by using the asymptotic representation of the Laplacian in \eqref{Laplaciano.formula.asintotica}, it follows that  
\[
u(x)\leq\dashint_{B_\varepsilon (x)} u(y)\,dy-\varepsilon^2\frac{n}{2(n+2)}(f(x))^{1/n}+o(\varepsilon^2)\qquad\textrm{as} \ \varepsilon\to0,
\]
which improves on \eqref{MVF.intro.convexity}.
Even more, we can replace the average over the ball in the previous expression by the infimum over the family of averages of $u$ over ellipsoids that are volume-preserving affine deformations of $B_\varepsilon (x)$. 
Motivated by  equality \eqref{first.main.ingredient.intro.88} we can expect an equality to hold for this infimum. 
This is in fact the main contribution of this article.

The main results of this paper are asymptotic mean value properties for the Monge-Amp\`ere equation that hold point-wise in the $C^2$ case,
Theorem~\ref{theorem.C2.case} and Theorem~\ref{theorem.C2.case2}, and in the viscosity sense in the general case, Theorem~\ref{thm.MVP.visco} and Theorem~\ref{thm.MVP.visco2}.
Our first two results give asymptotic mean value formulas when we consider ellipsoids that do not become too degenerate near the point under consideration. 
\begin{theorem}\label{theorem.C2.case} Let   $\phi(\varepsilon)$ 
be a positive function 
such that 
\begin{equation}\label{hipotesis.phi}
 \lim_{\varepsilon\to0} \phi(\varepsilon)=\infty
 \qquad
  \textrm{and}
 \qquad
 \lim_{\varepsilon\to0} \varepsilon\,\phi(\varepsilon) =0.
\end{equation}
 Let $u$ be convex and $C^2$ in $\Omega$. 
 Then, for every $x\in\Omega$ we have 
 \begin{equation}\label{absolute.goal.C2}
\begin{split}
u(x)=
\mathop{\mathop{\inf}_{\det A=1}}_
{A\leq \phi(\varepsilon)I}
\bigg\{
\dashint_{B_{\varepsilon}(0)}
u(x+Ay)
\,dy
\bigg\}
-
\frac{n}{2(n+2)}\,
\left(\det{D^2u(x)}\right)^{1/n}
\varepsilon^2
+
o(\varepsilon^2),
\end{split}
\end{equation}
 as $\varepsilon\to0$.
\end{theorem}

Observe that it follows from \eqref{hipotesis.phi} that the ellipsoids in \eqref{absolute.goal.C2}  shrink to $x$ as $\varepsilon\to0$. \par
Notice that Theorem \ref{theorem.C2.case} shows that a 
smooth ($C^2$) convex function is a classical solution to the Monge-Amp\`ere equation $\det D^2u=f$ in $\Omega$ if and only if
the asymptotic mean value formula
$$
u(x)=
\mathop{\mathop{\inf}_{\det A=1}}_
{A\leq \phi(\varepsilon)I}
\bigg\{
\dashint_{B_{\varepsilon}(0)}
u(x+Ay)
\,dy
\bigg\}
-
\frac{n}{2(n+2)}\,
\left(f(x)\right)^{1/n}
\varepsilon^2
+
o(\varepsilon^2),
$$
as $\varepsilon\to 0$ holds point-wise for $x \in \Omega$.

Our next result says that this mean value property extends to viscosity solutions but this time the mean value property need also
to be understood  in the  viscosity sense.

\begin{theorem}\label{thm.MVP.visco}
Let $\phi$ be a positive function satisfying \eqref{hipotesis.phi}, and $f\in C(\Omega)$ a non-negative function.
Then,
a convex function $u\in C(\Omega)$ is a viscosity  subsolution (respectively, supersolution) of the Monge-Amp\`ere equation 
\begin{equation}\label{MVP.solid.M-A.visco}
\det D^2u=f\quad \textrm{in}\ \Omega
\end{equation}
if and only if 
\begin{equation}\label{MVP.solid.visco}
\begin{split}
u(x)
\leq
\mathop{\mathop{\inf}_{\det A=1}}_
{A\leq \phi(\varepsilon)I}
\bigg\{
\dashint_{B_{\varepsilon}(0)}
u(x+Ay)
\,dy
\bigg\}
-\frac{n}{2(n+2)}\,(f(x))^{1/n}\,\varepsilon^2
+o(\varepsilon^2)
\end{split}
\end{equation}
as $\varepsilon\to0$ for $x\in \Omega$
(respectively, $\geq$) in the viscosity sense;
meaning that whenever a convex paraboloid $P$ touches $u$ from above (respectively below) at $x$, the mean value formula \eqref{MVP.solid.visco} is satisfied for the function $P$ at the point $x$.
\end{theorem}

The restriction $A\leq \phi(\varepsilon)I$ imposed to the set of matrices over which the infimum is taken in \eqref{absolute.goal.C2} makes the formula local.
For every $x\in\Omega$, the conditions  $A\leq \phi(\varepsilon)I$  and $|y|\leq \varepsilon$ imply that $x+Ay\in\Omega$ for $\varepsilon$ small enough; 
in fact, we have
\[
\textrm{dist}(x+Ay,x)=|Ay|\leq  \varepsilon\, \phi(\varepsilon)  \leq \textrm{dist}(x,\partial\Omega)
\]
for $\varepsilon$ sufficiently small (since $\varepsilon\,\phi(\varepsilon)\to 0$ as $\varepsilon \to 0$).
As an example of a function $\phi$ that verifies \eqref{hipotesis.phi} we mention $\phi(\varepsilon)=\varepsilon^{-\alpha}$ for $0<\alpha<1$.

Strictly convex paraboloids satisfy  \eqref{absolute.goal.C2} without the
remainder $o(\varepsilon^2)$.
However, when $\lambda=0$ is an eigenvalue of the Hessian matrix of the paraboloid, the remainder 
$o(\varepsilon^2)$ appears due to the restriction $A\leq \phi(\varepsilon) I$;
see Example \ref{example.paraboloids.what.happens.with.zero.eigs} below for details.

Next,  we present a different mean value formula where we do not impose the condition $A\leq \phi(\varepsilon)I$.
The statement for this formula is more natural and for convex paraboloids it holds without the remainder.
On the other hand, this formula is not local since the paraboloids need not shrink to $x$
as $\varepsilon \to 0$. 

In our formula we consider averages of $u$ over ellipsoids depending on $A$ and $\varepsilon$.
We must require that this ellipsoid is  included in the domain of definition of the function $u$.
Therefore we restrict the infimum over matrices $A$ for which the ellipsoid
\[
E_\varepsilon (A,x) =
\big\{y\in\mathbb{R}^n:\  |A^{-1}(y-x)|\leq \varepsilon\big\}=
\big\{x+Ay:\  y\in B_{\varepsilon}(0)\big\}
\]
 is contained in $\Omega$.

\begin{theorem}\label{theorem.C2.case2} 
 Let $u$ be convex and $C^2$ in $\Omega$. 
 Then, for every $x\in\Omega$
 \begin{equation}\label{absolute.goal.C2.2}
\begin{split}
u(x)=
\mathop{\mathop{\inf}_{\det A=1}}_
{E_\varepsilon (A,x) \subset\Omega}\bigg\{
\dashint_{B_{\varepsilon}(0)}
u(x+Ay)
\,dy
\bigg\}
-
\frac{n}{2(n+2)}\,
\left(\det{D^2u(x)}\right)^{1/n}
\varepsilon^2
+
o(\varepsilon^2),
\end{split}
\end{equation}
asymptotically as $\varepsilon\to0$.
\end{theorem}

As it was earlier the case, we have that the mean value formula holds point-wise for smooth functions and characterizes smooth
solutions to Monge-Amp\`ere, but when we want to extend the characterization for viscosity solutions the formula
has to be interpreted in the viscosity sense.

\begin{theorem}\label{thm.MVP.visco2}
Let $f\in C(\Omega)$ be a non-negative function.
Then,
a convex function $u\in C(\Omega)$ is a viscosity  subsolution (respectively, supersolution) of the Monge-Amp\`ere equation 
\begin{equation}\label{MVP.solid.M-A.visco.88}
\det D^2u=f\quad \textrm{in}\ \Omega
\end{equation}
if and only if 
\begin{equation}\label{MVP.solid.visco2}
\begin{split}
u(x)
\leq
\mathop{\mathop{\inf}_{\det A=1}}_
{E_\varepsilon (A,x) \subset\Omega}\bigg\{
\dashint_{B_{\varepsilon}(0)}
u(x+Ay)
\,dy
\bigg\}
-\frac{n}{2(n+2)}\,(f(x))^{1/n}\,\varepsilon^2
+o(\varepsilon^2)
\end{split}
\end{equation}
as $\varepsilon\to0$
(respectively, $\geq$) in the viscosity sense for $x\in \Omega$;
meaning that whenever a convex paraboloid $P$ touches $u$ from above (respectively below) at $x$, the mean value formula \eqref{MVP.solid.visco2} is satisfied for the function $P$ at the point $x$.
\end{theorem}

Notice that we always have
\begin{equation}\label{eq.remark.different.infima}
\begin{split}
\inf_{\det A=1}
\dashint_{B_{\varepsilon}}
u(x+Ay)
\,dy
\leq
\mathop{\mathop{\inf}_{\det A=1}}_
{A\leq \phi(\varepsilon)I}
\dashint_{B_{\varepsilon}}
u(x+Ay)
\,dy.
\end{split}
\end{equation}
The equality holds for strictly convex paraboloids and for smooth functions the difference is of order $o(\varepsilon^2)$.
The difference  is not of order $o(\varepsilon^2)$ for viscosity solutions, see Example~\ref{Ex-MVP-viscosa} below. 
Therefore, Theorem~\ref{theorem.C2.case} and Theorem~\ref{theorem.C2.case2} are different in nature, they do not imply each other in an obvious way.
The infimum over ellipsoids with the restriction $A\leq \phi(\varepsilon)I $
is bigger or equal than the   infimum over ellipsoids contained in $\Omega$ for $\varepsilon$ small enough
(because there are more ellipsoids contained in $\Omega$ than  satisfying the  restriction $A\leq \phi(\varepsilon)I$). 
To check whether $u$ is a subsolution of equation \eqref{MVP.solid.M-A.visco.88}  the upper bound for  $u(x)$ with restricted ellipsoids is more convenient. 
On the other hand, the lower bound for $u(x)$ using ellipsoids contained in $\Omega$ is
more convenient  when checking whether $u(x)$ is a supersolution of the same equation. 
\par
Several remarks are in order.
\par
\begin{remark}
Analogous formulas can be obtained by considering the integral over the boundary of the ellipsoids.
That is, for $u$ convex and $C^2$ in $\Omega$ we have 
\begin{equation}\label{MVP.surface}
\begin{split}
u(x)
&=
\mathop{\mathop{\inf}_{\det A=1}}_
{E_\varepsilon (A,x) \subset\Omega}\bigg\{
\dashint_{\partial B_{\varepsilon}(0)}
u(x+Ay)
\,d\mathcal{H}^{n-1}(y)
\bigg\}
-\frac{1}{2}\,\left(\det{D^2u(x)}\right)^{1/n}\,{\varepsilon}^2
+o(\varepsilon^2)\\
&=
\mathop{\mathop{\inf}_{\det A=1}}_
{A\leq \phi(\varepsilon)I}
\bigg\{
\dashint_{\partial B_{\varepsilon}(0)}
u(x+Ay)
\,d\mathcal{H}^{n-1}(y)
\bigg\}
-\frac{1}{2}\,\left(\det{D^2u(x)}\right)^{1/n}\,{\varepsilon}^2
+o(\varepsilon^2)
\end{split}
\end{equation}
as $\varepsilon\to0$.
See Remark~\ref{proof.border.formula} below. % for some comments on the proof.
These formulas also hold in the viscosity sense  and they characterize viscosity solutions 
to the Monge-Amp\`ere equation.
\end{remark}

\begin{remark}\label{remark.polar.decomposition}
 We  assume  without loss of generality that the matrices $A$ in Theorems \ref{theorem.C2.case} and \ref{thm.MVP.visco}, 
and in general
 throughout this paper are symmetric and positive definite. This follows from the (unique) left polar decomposition of $A$, namely $A = SQ$,
where $Q$ is orthogonal and $S$ is a positive definite symmetric matrix. In fact, it is easy to see that $S=(AA^t)^{1/2}$ and $Q=S^{-1}A$.
\end{remark}

\begin{remark} \label{remark18} Assume that $u\in C^2(\Omega)$ and that all the eigenvalues of $D^2u$ are positive (notice that in this case $u$ is strictly convex). 
For every $x\in\Omega$,  the matrix
\[
A^*=\left(\det{D^2u(x)}\right)^\frac{1}{2n}D^2u(x)^{-\frac{1}{2}}
\]
has determinant equal to 1 and satisfies $A^*\leq\phi( \varepsilon)I$ for $\varepsilon$ small enough. Therefore, for $\varepsilon$ small we have
\[
\begin{split}
n\det(D^2u(x))^{1/n}
&=
{\rm trace}((A^*)^tD^2u(x)A^*)
\\
&=
 \inf_{\det A=1}{\rm trace}\big(A^tD^2u(x)A\big)
\\
&=
 \inf_{\det A=1,\
A\leq \phi(\varepsilon)} 
 {\rm trace}\big(A^tD^2u(x)A\big),
\end{split}
\]
by Lemma \ref{lemma.uniform.ellipticity.visco} below. 
Moreover, $A^*$ realizes the infimum in \eqref{absolute.goal.C2} asymptotically; that is, making the choice $A=A^*$ in the argument leading to \eqref{proof.C2.case.aux.for.remark.in.intro} in the proof of Theorem \ref{theorem.C2.case}, plus 
\eqref{C2.case.second.part},
gives
\[
u(x)
=
\dashint_{B_{\varepsilon}(0)}
u(x+A^*y)\,dy
-\frac{n}{2(n+2)}\,\left(\det{D^2u(x)}\right)^{1/n}\,\varepsilon^2
+o(\varepsilon^2)
\]
for $\varepsilon$ small enough. Then, comparing with \eqref{absolute.goal.C2},
we get
\[
\inf_{\det A=1,\
A\leq \phi(\varepsilon) I } \bigg\{
\dashint_{B_{\varepsilon}(0)}
u(x+Ay)\,dy
\bigg\}
=
\dashint_{B_{\varepsilon}(0)}
u(x+A^*y)\,dy
+o(\varepsilon^2),
\]
as desired.
\end{remark}

After proving the main theorems we present several examples.
In Examples \ref{equaltyforallmatrices} and \ref{example.paraboloids.what.happens.with.zero.eigs} we study the behavior of the mean value formulas for paraboloids.
Example \ref{example.fourth.order.negative} shows that the hypothesis of convexity is needed.
In Example \ref{example4.4} we present a viscosity solution to the Monge-Amp\`ere equation that is not 
classical, but for which the asymptotic mean value property holds in the point-wise sense. 
Finally, Examples \ref{Ex-MVP-viscosa} and \ref{ejemplopablo} provide examples of a viscosity solution to the Monge-Amp\`ere 
equation such that the mean value property does not hold point-wise, although
it holds in the viscosity sense by Theorem \ref{thm.MVP.visco}.

Finally, we want to include a discrete mean value formula for Monge-Amp\`ere. This is similar to formula \eqref{Asymptotic.MVP.Harmonic.discrete}
for harmonic functions. 

\begin{theorem}\label{eq.mena.value.charact} 
Let $u$ be a 
convex
function in a domain $\Omega\subset\mathbb{R}^n$ and let $\phi$ be 
a positive function satisfying \eqref{hipotesis.phi}.  Let us denote by $\mathbb{O}$ the set of all orthonormal bases $V=\{v_1, \ldots, v_n\}$ of
$\mathbb{R}^n$. For $\varepsilon>0$ we define the sets  
 \[
I_\varepsilon^n=
\Big\{(\alpha_1, \dots, \alpha_n ) \in\mathbb{R}^n:
\prod_{j=1}^n \alpha_{j}=1
\quad \text{ and }\quad 
0<\alpha_{j}< \phi^2(\varepsilon) \Big\}.
\]
We have that the 
asymptotic expansion 
$$
\displaystyle u (x) = \inf_{ V\in \mathbb{O}}
\ \inf_{\alpha_{i}\in I_\varepsilon^n} \left\{ \frac{1}{n}
 \sum_{i=1}^n \frac12 u (x + \varepsilon \sqrt{\alpha_i} v_i) + \frac12 u (x - \varepsilon \sqrt{\alpha_i} v_i) 
\right\}  -\frac{\varepsilon^2}{2} (f(x))^{1/n} + o(\varepsilon^2)
$$
as $\varepsilon \to 0$,
holds in the viscosity sense
if and only if $u$ is a solution to the Monge-Amp\`ere equation
$$
\det (D^2 u (x)) = f(x)
$$
in the viscosity sense. \end{theorem}

Note the similarity with the mean value formula \eqref{MVP.solid.visco}
in the sense that $\sqrt{\alpha_i}$ can be regarded as analogous to the eigenvalues of $A$ (and hence the condition $\prod_i \alpha_i =1$
is analogous to $\det (A) =1$) while the orthonormal basis $\{v_i\}$ is analogous to  a basis of eigenvectors of $A$. 

\medskip

The paper is organized as follows. In Section \ref{sect.not.and.prelim} we provide some necessary preliminaries and information 
about  general viscosity theory for Monge-Amp\`ere. In 
Section \ref{section.viscosity.case} we provide the proofs of Theorems \ref{theorem.C2.case},  \ref{thm.MVP.visco}, \ref{theorem.C2.case2} and  \ref{thm.MVP.visco2}.
In Section \ref{sect.examples} we discuss some particular examples. 
Finally, in Section \ref{sect.another.MVF} we prove the discrete version of the mean value formula.

\section{Notation and preliminaries}\label{sect.not.and.prelim}

In this section we set the notation, recall basic results on the Monge-Amp\`ere equation,  and state some definitions.

For symmetric square matrices, $A>0$ means positive definite and $A\geq0$ means positive semidefinite. We will denote $\lambda_i(A)$ the eigenvalues of $A$, in particular $\lambda_{\min}(A)$ and $\lambda_{\max}(A)$ are the smallest and largest eigenvalues, respectively.

We will denote the $k$-dimensional ball of radius 1 and center 0 by $B_1^{(k)}(0)=\{x\in\mathbb{R}^k:\ |x|\leq1\}$  and  the corresponding $(k-1)$-dimensional sphere by $\partial B_1^{(k)}(0)=\{x\in\mathbb{R}^k:\ |x|=1\}$. Whenever $k$ is clear from the context, it will be omitted.
%Whenever the dimension $k$ is clear from context, we will simply write $B_1(0)$ and $\partial B_1(0)$.
We use the notation $\mathcal{H}^{k}$ for the $k$-dimensional Hausdorff measure.

A $C^2$ function $P(x)$ is a paraboloid if and only if it coincides with its second order Taylor expansion, i.e., we have
\[
P(x)=P(x_0)+\langle\nabla P(x_0),(x-x_0)\rangle+\frac12\langle D^2 P(x_0)(x-x_0),(x-x_0)\rangle
\]
for any given $x_0$. Furthermore, $P(x)$ is a convex paraboloid if and only if $D^2P\geq0$.

\medskip

We recall from \cite{Caf90a, Figalli.2017, Gutierrez} the notion of viscosity solution used in the sequel, see  also \cite{Caf90b, Caf91, Caf93, 
lions, lions2,Trudinger3, Trudinger2} for additional references on Monge-Amp\`ere equations.

\begin{definition}
Let $u\in C(\overline\Omega)$ be a convex  function and $f\in C(\Omega)$, $f\geq0$. The function  $u$ is a viscosity supersolution (subsolution) of the Monge-Amp\`ere equation
\begin{equation}\label{blacksabbath}
\det(D^2u)=f\qquad\text{in}\ \Omega,
\end{equation}
if for every
convex paraboloid $P$
 that touches $u$ from below at $x_0\in\Omega$  (respectively from above) we have
\[
\det(D^2P(x_0))\leq f(x_0) \quad\textrm{(respectively $\geq f(x_0)$)}.
\]
\end{definition}

\begin{remark}\label{conceptos.solucion.MA}
There are other notions of generalized solution of  \eqref{blacksabbath} in the literature.
The notion of \emph{admissible weak solution} (see \cite{Trudinger2}) is equivalent to the notion of  \emph{generalized solution in the sense of  Aleksandrov} (see \cite[Section 1.2]{Gutierrez} and also \cite{Aleksandrov,Bakelman,Caf90a,Pogorelov}). When $f\in C(\Omega)$, a convex function $u\in C(\overline\Omega)$, solution of \eqref{blacksabbath} in the Aleksandrov sense (or admissible weak sense) is a viscosity solution; the converse is true whenever $f>0$ in $\overline\Omega$ (see \cite[Section 2]{Caf90a} and \cite[Propositions 1.3.4 and 1.7.1]{Gutierrez}).
\end{remark}

There is considerable work in the literature establishing existence, uniqueness and regularity of solutions to \eqref{problem.intro.model}, see \cite{ Ca-Ni-Sp, Ca-Ni-Sp2, Figalli.2017, Gutierrez, Trudinger.Wang} and the references therein. The main ingredients of this theory are summarized as follows:

\noindent1. The Monge-Amp\`ere operator, when written as $(\det (D^2u))^{1/n}$, is concave fully nonlinear. It can be expressed as an infimum of a family of linear operators as follows,
\begin{equation}\label{first.main.ingredient.intro}
 n\big(\det D^2u(x)\big)^{1/n}=\inf_{\det A=1} {\rm trace}(A^tD^2u(x)A)
\end{equation}
This characterization is a well-known, general property for non-negative matrices, proved below in  Lemma \ref{caract.determ} for the sake of completeness.
Notice that the operator is degenerate since the condition $\det(A)=1$ allows the matrix $A$ to have arbitrarily small eigenvalues.
This idea has been used in \cite{Gaveau} to obtain an interpretation of Monge-Amp\`ere in terms of optimal control.

\noindent2. The fact that $\det D^2u$ can be represented as a concave fully nonlinear operator implies that pure second derivatives are subsolutions of an equation with bounded measurable coefficients and as such, are bounded from above. For that purpose, the boundary and data must be smooth and the 
domain strictly convex.  In fact, the geometry of the domain is an important issue in the regularity theory for degenerate operators depending on the eigenvalues of the Hessian, see \cite{Ca-Ni-Sp2}.

 \noindent3. The last ingredient of the theory is that for a convex solution of the equation  with $f$ strictly positive, the linearized  operator is uniformly elliptic  and general regularity theory applies. In particular, the Evans-Krylov theorem implies that solutions are $C^{2,\alpha}$ and from there, as smooth as two derivatives better than $f$.

Let us  state precisely what we understand by \lq\lq satisfying the mean value 
property in a viscosity sense\rq\rq. First, recall that  given  a constant $c$ and a real function $g$ 
$$c\le g(\varepsilon) + o(\varepsilon^2)\text{ as } \varepsilon\to 0$$ whenever we have
$$\lim_{\varepsilon\to 0} \frac{\left[ c-g(\varepsilon)\right]^+}{\varepsilon^2}=0, $$
and 
$$c\ge g(\varepsilon) + o(\varepsilon^2)\text{ as } \varepsilon\to 0$$ whenever we have
$$\lim_{\varepsilon\to 0} \frac{\left[ c-g(\varepsilon)\right]^-}{\varepsilon^2}=0.$$

\begin{definition} \label{def.asymp.mean} Let  $\phi$ be a positive function satisfying \eqref{hipotesis.phi}. A function $u$ verifies the mean value formula
$$
u(x)
=
\mathop{\mathop{\inf}_{\det A=1,}}_{A\leq \phi(\varepsilon)I} \bigg\{
\dashint_{B_{\varepsilon}(0)}
u(x+Ay)\,dy
\bigg\}
-\frac{n}{2(n+2)}\,(f(x))^{1/n}\,\varepsilon^2
+o(\varepsilon^2)
$$
as $\varepsilon\to0$
in the viscosity sense if for every convex paraboloid $P$ that touches $u$ from above (respectively from below) at $x$, it holds that
$$
P(x)
\leq (\geq)
\mathop{\mathop{\inf}_{\det A=1,}}_{A\leq \phi(\varepsilon)I} \bigg\{
\dashint_{B_{\varepsilon}(0)}
P(x+Ay)\,dy
\bigg\}
-\frac{n}{2(n+2)}\,(f(x))^{1/n}\,\varepsilon^2
+o(\varepsilon^2)
$$
as $\varepsilon\to0$.
\end{definition}

 An analogous definition holds when the   infimum is taken over ellipsoids contained in $\Omega$. In fact, ellipsoids play a very important role in the geometry of Monge-Amp\`ere.
Let us recall the notation used in Theorems \ref{theorem.C2.case2} and \ref{thm.MVP.visco2}.

\begin{definition}\label{def.ellipsoids}
Given a real number $\varepsilon >0$, a point $x\in\mathbb{R}^n$ and   a matrix  $A$ with $\det A=1$, the set 
\begin{equation}\label{definition.ellipsoids}
\begin{split}
E_\varepsilon (A,x) &=
\left\{y\in\mathbb{R}^n:\  |A^{-1}(y-x)|\leq \varepsilon  \right\}\\
&=\left\{y\in\mathbb{R}^n:\ \langle (AA^t)^{-1}(y-x),(y-x)\rangle\leq \varepsilon^2\right\}\\
&=\big\{x+Ay:\  y\in B_{\varepsilon}(0)\big\}.
\end{split}
\end{equation}
is an ellipsoid centered at $x$, with axes in the directions of
the  eigenvectors of $AA^t$, and axes lengths given by square roots of the corresponding eigenvalues, scaled by a factor $\varepsilon$.
\end{definition}

\begin{remark}\label{properties.elipsoids}
The following are some elementary but useful facts:
\begin{enumerate}

\item Whenever  $\varepsilon =1$ we will say that
the ellipsoid is unitary since then
\[
|E_1(A,x)|=\det( A)\,|B_1(0)|=|B_1(0)|.
\]

\item $E_\varepsilon (A,0)= \varepsilon \,E_1(A,0)$.

\item $|E_\varepsilon (A,x)|
=
\varepsilon^n |E_1(A,x)|
=
\varepsilon^n\,|B_1(0)|=|B_\varepsilon (0)|.$

\item The representation of a given ellipsoid by a matrix $A$ is not unique, since for any  $Q$ with $Q^tQ=I$, the matrix $QA$  represents the same ellipsoid. 

\item According to Remark \ref{remark.polar.decomposition}, we can always assume that an ellipsoid is given by a symmetric, positive definite matrix.

\end{enumerate}
\end{remark}

The following elementary fact will be used several times in the sequel. 
\begin{lemma}\label{lemma.trace.integral}
Let $M$ be a square matrix of dimension $n$. Then,
\begin{align}
{\rm trace}(M)
&=
\frac{n}{\varepsilon^2}
\dashint_{\partial B_\varepsilon (0)}
\langle My,y\rangle
\,d\mathcal{H}^{n-1}(y),
\label{trace.surface.integral.representation}
\\
&=
\frac{n+2}{\varepsilon^2}
\dashint_{B_\varepsilon (0)}
\langle My,y\rangle
\,dy,
\label{trace.solid.integral.representation}
\end{align}
\end{lemma}

\begin{proof} (Included for the sake of completeness)
 Expression \eqref{trace.surface.integral.representation} follows immediately from 
the divergence theorem, just realizing that  the exterior unit normal to $\partial B_\varepsilon (0)$ at  $y\in \partial B_\varepsilon(0)$ is $y/\varepsilon$. For the second part, a change to polar coordinates and \eqref{trace.surface.integral.representation} (with $\varepsilon=1$) yield
\[
\begin{split}
\dashint_{B_{\varepsilon}(0)}
\langle
My,y
\rangle
\,dy
&=
\frac{1}{|B_{\varepsilon}(0)|}
\int_{0}^{\varepsilon}\int_{\partial B_{1}(0)}
\langle
My,y
\rangle
\,d\mathcal{H}^{n-1}(y)\,
r^{n+1}
\,dr
\\
&=
\frac{\varepsilon^{2}}{(n+2)|B_{1}(0)|}
\int_{\partial B_{1}(0)}
\langle
My,y
\rangle
\,d\mathcal{H}^{n-1}(y)
=
\frac{\varepsilon^{2}}{n+2}\,
{\rm trace}(M).\qedhere
\end{split}
\]
\end{proof}

We conclude this section with the following linear algebra fact, which is the key to what follows. 

\begin{lemma}\label{caract.determ}
Let $B$ be a symmetric and positive semidefinite matrix. Then,
\[
\inf_{\det A=1}{\rm trace}(A^tBA)=n (\det(B))^{1/n}.
\]
On the other hand, if $B$ has a negative eigenvalue, then the  infimum is $-\infty$.
\end{lemma}

\begin{proof}
Let  $A$ with $\det A=1$, which we can assume symmetric and positive definite by Remark~\ref{remark.polar.decomposition}. The matrix $A^tBA$ is symmetric and positive semidefinite and the arithmetic--geometric inequality gives,
\[
\det(B)^{1/n}
=
\Big(\prod_{i=1}^{n} \lambda_i(A^tBA)\Big)^{1/n}
\leq
\frac{1}{n}\sum_{i=1}^{n} \lambda_i(A^tBA)
=
\frac{{\rm trace}(A^tBA)}{n}.
\]
As this is true for any $A$ with $\det A=1$, we deduce,
\[
n (\det(B))^{1/n}\leq\inf_{\det A=1}{\rm trace}(A^tBA).
\]

To derive the converse inequality, assume first that $B>0$. 
Then, the matrix $A=\det(B)^\frac{1}{2n}B^{-\frac{1}{2}}$ has determinant equal to 1 and verifies the equality,
$${\rm trace}(A^tBA)=n(\det(B))^{1/n}.$$

Let us now consider the case when $B\geq0$. Since the result is trivial when $B=0$, we can assume that 0 is an eigenvalue of the matrix  $B$ with multiplicity $n-k<n$, that is,
\[
B=P_B
\,{\rm diag}
\big(
\lambda_1(B),\ldots,\lambda_k(B),0,\ldots,0
\big)\,
P_B^t
\]
with $P_BP_B^t=I$. Fix $\delta>0$ and  define 
\begin{equation}\label{A_epsilon}
A_{\delta}=P_B\,{\rm diag}(\lambda_i(A_{\delta}))\,P_B^t
\end{equation}
with $P_B$ the same  as before and ${\rm diag}(\lambda_i(A_{\delta}))$ the diagonal matrix with entries
\[
\lambda_i(A_{\delta})=
\left\{\begin{split}
&\left(\frac{\delta}{{\rm trace}(B)}\right)^\frac12\hspace{38pt}{\rm for}\ i=1,\ldots ,k\\
&\left(\frac{\delta}{{\rm trace}(B)}\right)^{- \frac{k}{2(n-k)}}\quad{\rm for}\ i=k+1,\ldots ,n.
\end{split}
\right.
\]
In this way, $A_{\delta}$ is  positive definite with  $\det (A_{\delta})=1$, and
\[
{\rm trace}( A_{\delta}^tBA_{\delta})=\sum_{i=1}^k \lambda_i(A_{\delta})^2\lambda_i(B)=\delta.
\]
Since $\delta>0$ is arbitrary, we conclude that
\[
\inf_{\det A=1}{\rm trace}(A^tBA)=0=n(\det(B))^{1/n}.
\]

Finally, whenever $B$ has negative eigenvalues, one can adapt the above argument and write
\[
B=P_B
\,{\rm diag}
\big(
\lambda_1(B),\ldots,\lambda_k(B),\lambda_{k+1}^-(B),\ldots,\lambda_n^-(B)
\big)\,
P_B^t
\]
where $\lambda_i^-(B)$ denote the negative eigenvalues of $B$ (note that we are assuming that $B$ has $k$ non-negative eigenvalues). Then, choosing $A_\delta$ as in \eqref{A_epsilon} with
\[
\lambda_i(A_{\delta})=
\left\{\begin{split}
&\delta^\frac{1}{2}\hspace{38pt}{\rm for}\ i=1,\ldots ,k\\
&\delta^{- \frac{k}{2(n-k)}}\quad{\rm for}\ i=k+1,\ldots ,n.
\end{split}
\right.
\]
we get
\[
{\rm trace}( A_{\delta}^tBA_{\delta})
=
\delta \sum_{i=1}^k \lambda_i(B)
+
\delta^{- \frac{k}{n-k}} \sum_{i=k}^{n}\lambda_i^-(B).
\]
Taking $\delta\to0$ in the above expression, we obtain that $$\inf_{\det A=1} \Big\{{\rm trace}(A^tBA) \Big\}=-\infty.$$
\end{proof}

\section{Proof of the main results: Theorems \ref{theorem.C2.case},   \ref{thm.MVP.visco}, \ref{theorem.C2.case2}, and
 \ref{thm.MVP.visco2}}
 \label{section.viscosity.case}

In this section, we first prove  the mean value property under the restriction $A\le \phi(\varepsilon) I$ for solutions of the Monge-Amp\`ere equation, in the classical  and  viscosity cases, Theorems \ref{theorem.C2.case} and \ref{thm.MVP.visco}, respectively.
We start by showing that for strictly convex functions with positive definite Hessian  the matrices that compete for the infimum at a given point cannot be too degenerate. 

 \begin{lemma}\label{lemma.uniform.ellipticity.visco}
 Let $ u\in C^2(\Omega)$ with $\lambda_{\min} (D^2u) >0$. Then, for every $x\in\Omega$ and every
 \begin{equation}  \label{lemma.uniform.ellipticity.visco.theta0}
\theta> \theta_0:=\left(\frac{\Delta u(x)}{\lambda_{\min}\big(D^2 u(x)\big)}\right)^{1/2},
 \end{equation}
 we have
 \begin{equation}  \label{lemma.uniform.ellipticity.visco.statement}
 \inf_{\det A=1}{\rm trace}\big(A^tD^2 u(x)A\big)
 =
  \inf_{\det A=1,
 \ A\leq \theta I} {\rm trace}(A^tD^2 u(x)A) .
 \end{equation}
 \end{lemma}

 \begin{proof}
Let $A$ be a symmetric matrix 
%(see Remark~\ref{remark.polar.decomposition}) 
with $\det(A)=1$, and write $A=QJQ^t$ with $Q$ orthogonal. Then, we have
\[
 \begin{split}
 \textrm{trace}(A^tD^2 u(x)A)&=\sum_{i=1}^{n}\lambda_i^2(A)(Q^tD^2 u(x)Q)_{ii}\\
 &\geq \lambda_{\min}\big(D^2 u(x)\big)\sum_{i=1}^{n}\lambda_i^2(A) \\
 & \geq \lambda_{\min}\big(D^2 u(x)\big)\,
 \lambda_{\max}^{2}(A).
 \end{split}
\]
Therefore, we conclude that
 \begin{equation}\label{estimate1.main.proof}
  \mathop{\mathop{\inf}_{\det A=1,}}_{\lambda_{\max}(A)>\theta}
 \textrm{trace}(A^tD^2 u(x)A) \geq
 \lambda_{\min}\big(D^2 u(x)\big)\,
\theta^2.
 \end{equation}
 Moreover, by
  choosing $A=I$,  we have
 \begin{equation}\label{estimate2.main.proof}
 \mathop{\mathop{\inf}_{\det A=1}}\textrm{trace}(A^tD^2 u(x)A)
 \leq \Delta  u(x).
 \end{equation}
From \eqref{estimate1.main.proof} and \eqref{estimate2.main.proof}, whenever
 $\theta
 >
 \theta_0
 $ for $\theta_0$
 given by \eqref{lemma.uniform.ellipticity.visco.theta0}, we have that
 \[
  \mathop{\mathop{\inf}_{\det A=1,}}_{\lambda_{\max}(A)>\theta }
 \textrm{trace}(A^tD^2 u(x)A)
\\
 >
 \inf_{\det A=1}
 \textrm{trace}(A^tD^2 u(x)A).
 \]
 This proves the lemma, since
\[
 \inf_{\det A=1}
 \textrm{trace}(A^tD^2 u(x)A)
 =
 \min\Bigg\{
 \mathop{\mathop{\inf}_{\det A=1,}}_{\lambda_{\max}(A)> \theta} \textrm{trace}(A^tD^2 u(x)A),
 \mathop{\mathop{\inf}_{\det A=1,}}_{
 \ \lambda_{\max}(A)\leq \theta} \textrm{trace}(A^tD^2 u(x)A)\Bigg\}.\qedhere
\]
 \end{proof}

We can proceed with the proof of  Theorem \ref{theorem.C2.case}.

\begin{proof}[Proof of Theorem \ref{theorem.C2.case}]
Given $x\in\Omega$, let us define the paraboloid 
\[
P(z)=u(x)+\langle\nabla u(x),z-x\rangle+\frac12\langle D^2u(x)(z-x),(z-x)\rangle.
\]
 Since $u\in C^2(\Omega)$, we have
\[
u(z)-P(z)=o(|z-x|^2)\qquad\textrm{as}\ z\to x,
\]
which means that for every $\eta >0$, there exists $\delta>0$ such that for every $z\in B_\delta(x),$
\begin{equation} \label{u-P.o.pequena}
P(z)-\frac{\eta}{2}|z-x|^2\leq u(z)\leq P(z)+\frac{\eta}{2}|z-x|^2,
\end{equation}
 with equality only when $z=x$. For convenience, let us denote
\[
P_\eta^\pm(z)=P(z)\pm\frac{\eta}{2}|z-x|^2.
\]
Let us assume first that $D^2u(x)>0$.
Then,
\begin{equation}\label{proof.C2.case.aux.for.remark.in.intro}
\begin{split}
\dashint_{B_{\varepsilon}(0)}
&
\left(
P_\eta^\pm(x+Ay)-P_\eta^\pm(x)
\right)
\,dy
=
\frac12\dashint_{B_{\varepsilon}(0)}
\left(
\big\langle A^tD^2u(x)Ay,y\big\rangle\pm\eta |Ay|^2
\right)
\,dy
\\
&=
\frac12
\dashint_{B_{\varepsilon}(0)}
\left\langle A^t\left(D^2u(x)\pm\eta I\right)Ay,y\right\rangle
\,dy
=
\frac{\varepsilon^2}{2(n+2)}\,\textrm{trace}\left(A^t\left(D^2u(x)\pm\eta I\right)A\right),
\end{split}
\end{equation}
by Lemma \ref{lemma.trace.integral}.
From here, we can use Lemmas \ref{lemma.uniform.ellipticity.visco} and \ref{caract.determ} to get
\begin{equation}\label{C2.case.first.part}
\begin{split}
\mathop{\mathop{\inf}_{\det A=1,}}_{
A\leq \phi(\varepsilon)I} \bigg\{
\dashint_{B_{\varepsilon}(0)}&
\left(
P_\eta^\pm(x+Ay)-P_\eta^\pm(x)
\right)
\,dy
\bigg\}
\\
&=
\frac{\varepsilon^2}{2(n+2)}\,
\mathop{\mathop{\inf}_{\det A=1,}}_{
A\leq \phi(\varepsilon)I}\bigg\{
\textrm{trace}\left(A^t\left(D^2u(x)\pm\eta I\right)A\right)
\bigg\}
\\
&=
\frac{\varepsilon^2}{2(n+2)}\,
\mathop{\mathop{\inf}_{\det A=1}}
\textrm{trace}\left(A^t\left(D^2u(x)\pm\eta I\right)A\right)
\\
&=
\frac{n}{2(n+2)}\,
\left(\det\left(D^2u(x)\pm\eta I\right)\right)^{1/n}
\varepsilon^2
\end{split}
\end{equation}
as long as
\[
\varepsilon <\frac{\lambda_{\min}(D^2u(x))-\eta}{\Delta u(x)+ n\eta},
\]
which holds for $\eta, \varepsilon$ small enough because $D^2u(x)>0$.

On the other hand, from \eqref{hipotesis.phi} it follows that  the conditions  $A\leq \phi( \varepsilon)I$  and $|y|\leq \varepsilon$ imply
 $x+Ay\in B_\delta(x)$  for $\varepsilon< \varepsilon_0$, where $\varepsilon\, \phi(\varepsilon)< \delta$ for every $\varepsilon<\varepsilon_0$ and $\delta$ is as 
 in \eqref{u-P.o.pequena}. Therefore, by \eqref{u-P.o.pequena}, if $\varepsilon<\varepsilon_0$, then
 \begin{equation}\label{C2.case.second.part}
P_\eta^-(x+Ay)
\leq
u(x+Ay)
\leq
P_\eta^+(x+Ay)\qquad\textrm{for every}\ y\in B_\varepsilon (0).
\end{equation}
Then, 
\eqref{C2.case.first.part} and \eqref{C2.case.second.part} give
\begin{equation}\label{C2.case.final.loop.first.part}
\begin{split}
\frac{n}{2(n+2)}\,&
\left(\det\left(D^2u(x)-\eta I\right)\right)^{1/n}
\varepsilon^2
\\
&\leq
\mathop{\mathop{\inf}_{\det A=1,}}_{
A\leq \phi(\varepsilon)I}\bigg\{
\dashint_{B_{\varepsilon}(0)}
\left(
u(x+Ay)-u(x)
\right)
\,dy
\bigg\}
\\
&\leq
\frac{n}{2(n+2)}\,
\left(\det\left(D^2u(x)+\eta I\right)\right)^{1/n}
\varepsilon^2.
\end{split}
\end{equation}

To conclude, we observe that
\begin{equation}\label{C2.case.determinant.estimates}
\left(\det\left(D^2u(x)\pm\eta I\right)\right)^{1/n}
=
\left(\det{D^2u(x)}\right)^{1/n}
+O(\eta)
\qquad
\textrm{as $\eta\to0$}.
\end{equation}
To see this, first notice that
\begin{equation}\label{det.expansion.formula}
\det\left(D^2u(x)\pm\eta I\right)
=
\det{D^2u(x)}
+
\sum_{k=1}^{n}(\pm \eta)^{k}\sigma_{n-k}\left(D^2u(x)\right),
\end{equation}
where the coefficients in the expansion are given by the 
elementary symmetric polynomials on the eigenvalues of $D^2u(x)$, which are positive by the convexity of $u$.
Therefore, 
\[
\left(
\det{D^2u(x)}-C\eta
\right)^{1/n}
\leq
\left(
\det\left(D^2u(x)\pm\eta I\right)
\right)^{1/n}
\leq
\left(
\det{D^2u(x)} + C\eta
\right)^{1/n}
\]
for some $C>0$.
Then, the Mean Value Theorem applied to $g(t)=t^{1/n}$, with $a=\det{D^2u(x)}$ and $b=\det{D^2u(x)}+C\eta$ gives that there 
exists $\xi\in(a,b)$ such that
\[
\left(
\det{D^2u(x)} + C\eta
\right)^{1/n}=
g(b)
=g(a)+g'(\xi)(b-a)
\leq
\left(
\det{D^2u(x)}
\right)^{1/n}
+C\eta.
\]
Similarly, we obtain that
\[
\left(
\det{D^2u(x)} - C\eta
\right)^{1/n}
\geq
\left(
\det{D^2u(x)}
\right)^{1/n}
-C\eta
\]
and \eqref{C2.case.determinant.estimates} follows.

Putting together \eqref{C2.case.final.loop.first.part} and \eqref{C2.case.determinant.estimates}
we have proved that for every $\eta>0$ there exists $\varepsilon_0$ such that for every $\varepsilon<\varepsilon_0$
\[
\begin{split}
\Bigg|
\mathop{\mathop{\inf}_{\det A=1,}}_{
A\leq \phi(\varepsilon)I} \bigg\{
\dashint_{B_{\varepsilon}(0)}
\left(
u(x+Ay)-u(x)
\right)
\,dy
\bigg\}
-
\frac{n}{2(n+2)}\,
\left(\det{D^2u(x)}\right)^{1/n}
\varepsilon^2
\Bigg|
\leq C\eta
\varepsilon^2,
\end{split}
\]
that is, \eqref{absolute.goal.C2} holds.

The case when $\lambda= 0$ is an eigenvalue of $D^2u(x)$  requires minor modifications of the above argument 
(we will assume $D^2u(x)\neq0$ because the result is trivial otherwise). In fact, the convexity of $u$ 
yields $u(x+Ay)\geq u(x)+\langle\nabla u(x),Ay\rangle$ and therefore,
\[
\begin{split}
\mathop{\mathop{\inf}_{\det A=1,}}_{
A\leq \phi(\varepsilon)I} \bigg\{
\dashint_{B_{\varepsilon}(0)}
\left(
u(x+Ay)-u(x)
\right)
\,dy
\bigg\}
\geq0
=
\frac{n}{2(n+2)}\,
\left(\det{D^2u(x)}\right)^{1/n}
\varepsilon^2,
\end{split}
\]
since the integral of the first-order term vanishes by symmetry. Thus, proving \eqref{absolute.goal.C2} amounts to show that
\begin{equation}\label{eignvalues.zero.leq}
\mathop{\mathop{\inf}_{\det A=1,}}_{
A\leq \phi(\varepsilon)I} \bigg\{
\dashint_{B_{\varepsilon}(0)}
\left(
u(x+Ay)-u(x)
\right)
\,dy
\bigg\}
\leq
o(\varepsilon^2),
\end{equation}
 as $\varepsilon \to0$. Fix $\eta>0$.  Arguing as before, we find that there exists $\varepsilon_0$ 
 such that  \eqref{C2.case.second.part} holds for every $\varepsilon<\varepsilon_0$ (we will only use the upper estimate). 
 Then, by  \eqref{C2.case.second.part},  and Lemmas  \ref{lemma.trace.integral}, \ref{lemma.uniform.ellipticity.visco} and \ref{caract.determ}
 we get
\[
\begin{split}
\mathop{\mathop{\inf}_{\det A=1,}}_{
A\leq \phi(\varepsilon)I} \bigg\{
\dashint_{B_{\varepsilon}(0)}&
\left(
u(x+Ay)-u(x)
\right)
\,dy
\bigg\}
\\
&\leq
\mathop{\mathop{\inf}_{\det A=1,}}_{
A\leq \phi(\varepsilon) I}\bigg\{
\dashint_{B_{\varepsilon}(0)}
\left(
P_\eta^+(x+Ay)-P_\eta^+(x)
\right)
\,dy
\bigg\}
\\
&=
\frac{\varepsilon^2}{2(n+2)}\,
\mathop{\mathop{\inf}_{ \det A=1,}}_{
A\leq \phi(\varepsilon) I} \bigg\{
\textrm{trace}\left(A^t\left(D^2u(x)+\eta I\right)A\right)
\bigg\}
\\
&=
\frac{\varepsilon^2}{2(n+2)}\,
\inf_{\det A=1}
\textrm{trace}\left(A^t\left(D^2u(x)+\eta I\right)A\right)
\\
&=
\frac{n}{2(n+2)}\,
\left(\det\left(D^2u(x)+\eta I\right)\right)^{1/n}
\varepsilon^2,
\end{split}
\]
as long as
\[
\varepsilon<\frac{\eta}{\Delta u(x)+ n\eta}.
\]
But them, \eqref{det.expansion.formula} gives
\[
0\leq\det\left(D^2u(x)+\eta I\right)
\leq
\det{D^2u(x)}+C\eta=C\eta
\]
and the proof of \eqref{eignvalues.zero.leq} is complete.
\end{proof}

Next,  we proceed with the proof of Theorem \ref{thm.MVP.visco}.

\begin{proof}[Proof of Theorem \ref{thm.MVP.visco}]
%We will show \eqref{MVP.solid.visco} since the proof of \eqref{MVP.surface.visco} is similar. 
Let us show first that  $u$ is a viscosity subsolution of \eqref{MVP.solid.M-A.visco.88} 
if and only if \eqref{MVP.solid.visco} holds
in the viscosity sense.  Assume that a convex paraboloid  $P$ touches $u$ from above at $x_0\in\Omega$. By definition of viscosity subsolution of \eqref{MVP.solid.M-A.visco.88}, we have
\begin{equation}\label{proof.visco.sub.0}
\det D^2 P(x_0)\geq f(x_0)
\end{equation}
and we want to show that
\begin{equation}\label{proof.visco.sub.1}
\begin{split}
u(x_0)
\leq
\mathop{\mathop{\inf}_{\det A=1,}}_{
A\leq\phi( \varepsilon) I} \bigg\{
\dashint_{B_{\varepsilon}(0)}
 P(x_0+Ay)\,dy
\bigg\}
-\frac{n}{2(n+2)}\,(f(x_0))^{1/n}\,\varepsilon^2
+o(\varepsilon^2)\qquad \textrm{as $\varepsilon\to0$}.
\end{split}
\end{equation}
But, in fact, Theorem \ref{theorem.C2.case} and \eqref{proof.visco.sub.0} show that
\begin{equation}\label{proof.visco.sub.2}
\begin{split}
\mathop{\mathop{\inf}_{\det A=1,}}_{
A\leq \phi(\varepsilon) I} \bigg\{
\dashint_{B_{\varepsilon}(0)}
\left(
P(x_0+Ay)-P(x_0)
\right)
\,dy
\bigg\}
& =
\frac{n}{2(n+2)}\,
\left(\det{D^2P(x_0)}\right)^{1/n}
\varepsilon^2
+
o(\varepsilon^2),
\\
&\geq
\frac{n}{2(n+2)}\,
(f(x_0))^{1/n}
\varepsilon^2
+
o(\varepsilon^2),
\end{split}
\end{equation}
as $\varepsilon\to0$,
which gives \eqref{proof.visco.sub.1}.

To prove the converse, let $P$ and $x_0$ be as before, and assume \eqref{proof.visco.sub.1}. Our goal is to show \eqref{proof.visco.sub.0}. Theorem \ref{theorem.C2.case}  and equation \eqref{proof.visco.sub.1} show that
\[
\begin{split}
\frac{n}{2(n+2)}\,(f(x_0))^{1/n}\, \varepsilon^2 +
o(\varepsilon^2)
&\leq
\mathop{\mathop{\inf}_{\det A=1,}}_{
A\leq \phi(\varepsilon)} \bigg\{
\dashint_{B_{\varepsilon}(0)}
\left(
 P(x_0+Ay)- P(x_0)
\right)
\,dy
\bigg\}
\\
&=
\frac{n}{2(n+2)}\,
\left(\det{D^2P(x_0)}\right)^{1/n}
\varepsilon^2
+
o(\varepsilon^2),
\end{split}
\]
as $\varepsilon \to0$. Then, dividing both sides by $\varepsilon^2$ and letting $\varepsilon\to0$ gives \eqref{proof.visco.sub.0}.

The supersolution case follows similarly.
\end{proof}

 Theorem \ref{theorem.C2.case2} follows from Theorem \ref{theorem.C2.case}. The key point in the argument is Lemma \ref{keylemma} below, which shows that when $\lambda_{\min}(D^2 u)>0$
 the  ellipsoids that compete for the infimum among ellipsoids contained in $\Omega$ must concentrate around the point where the mean value property is centered.
 Theorem \ref{thm.MVP.visco2}  in the viscosity case follows from 
 Theorem \ref{theorem.C2.case2} as Theorem \ref{thm.MVP.visco} follows from Theorem  \ref{theorem.C2.case}.

\begin{lemma} \label{keylemma}
Let $u\in C^2 (\Omega)$ be a convex function.
We fix $x_0\in\Omega$ and assume that $\lambda_{\min}(D^2 u(x_0))>0$.
For each $\varepsilon$, let $A_\varepsilon$ be a matrix with $\det A_\varepsilon=1$ that almost realizes the infimum; that is, such that
 \begin{equation}\label{absolute.goal.C88}
 \begin{array}{l}
 \displaystyle 
\dashint_{B_{\varepsilon}(0)}
\left(
u(x_0+A_\varepsilon y) - u(x_0) -
\frac{n}{2(n+2)}\,
\left(\det{D^2u(x_0)}\right)^{1/n}
\varepsilon^2
\right)\,dy  \\[10pt]
\quad  \displaystyle \leq \mathop{\mathop{\inf}_{\det A=1}}_
{E_{\varepsilon} (A,x_0) \subset\Omega}\bigg\{
\dashint_{B_{\varepsilon}(0)}
\left(
u(x_0+Ay) - u(x_0) -
\frac{n}{2(n+2)}\,
\left(\det{D^2u(x_0)}\right)^{1/n}
\varepsilon^2
\right)
\,dy
\bigg\} + o(\varepsilon^2).
% \\[10pt]
%\qquad  \displaystyle  \leq \dashint_{B_{\varepsilon_m}(0)}
%u(x_0+A_m y) - u(x_0)
%-
%\frac{n}{2(n+2)}\,
%\left(\det{D^2u(x_0)}\right)^{1/n}
%\varepsilon_m^2
%\,dy
\end{array}
\end{equation} 
Then, the ellipsoids 
$
E_{\varepsilon}(A_\varepsilon, x_0)
$
concentrate at the point $x_0$ as $\varepsilon \to 0$, in the sense that,
given $\delta>0$ there exists $\varepsilon_0>0$ such that
$$
E_{\varepsilon}(A_\varepsilon, x_0) \subset B_\delta (x_0)
$$
for each $\varepsilon<\varepsilon_0$.
\end{lemma}

\begin{proof} 
First, we observe that since $u$ is convex we can subtract  the supporting hyperplane of $u$ at $x_0$ and
assume  $u(x_0)=0$ and $u\geq 0$ in $\Omega$.
Since the eigenvalues  of $D^2 u(x_0)$ are strictly positive we have that there exists $c>0$ such that
\begin{equation} \label{real.madrid}
u(x) \geq c \, \delta^2, \qquad \mbox{for } |x - x_0| > \delta/2, 
\end{equation}
where $\delta$ can be chosen small. 

We argue by contradiction. Assume that there is $\delta>0$ and a sequence $A_m=A_{\varepsilon_m}$ such that
$\det A_m=1$, $E_{\varepsilon_m}(A_m, x_0)\subset \Omega$, 
$
E_{\varepsilon_m}(A_m, x_0)\cap (\Omega \setminus B_{\delta}(x_0)) \neq \emptyset
$
(the ellipsoids do not concentrate),  and
 \begin{equation}\label{absolute.goal.C88.99}
 \begin{array}{l}
 \displaystyle 
\dashint_{B_{\varepsilon_m}(0)}
\left(
u(x_0+A_my)  -
\frac{n}{2(n+2)}\,
\left(\det{D^2u(x_0)}\right)^{1/n}
\varepsilon_m^2
\right)
\,dy \\[10pt]
\quad  \displaystyle \leq  \mathop{\mathop{\inf}_{\det A=1}}_
{E_{\varepsilon_m} (A,x_0) \subset\Omega}\bigg\{
\dashint_{B_{\varepsilon_m }(0)}
\left(
u(x_0+Ay)  -
\frac{n}{2(n+2)}\,
\left(\det{D^2u(x_0)}\right)^{1/n}
\varepsilon_m^2
\right)
\,dy
\bigg\}+ o(\varepsilon_m^2 ).
% \\[10pt]
%\qquad  \displaystyle  \leq \dashint_{B_{\varepsilon_m}(0)}
%u(x_0+A_m y) 
%-
%\frac{n}{2(n+2)}\,
%\left(\det{D^2u(x_0)}\right)^{1/n}
%\varepsilon_m^2
%\,dy
\end{array}
\end{equation} 
Since
$u(x) \geq c \,\delta^2$  for  $|x - x_0| > \delta/2$
it holds that 
$$
\dashint_{B_{\varepsilon_m}(0)}
u(x_0+A_my) dy \geq \frac{C}{\varepsilon_m^n} \int_{\{y\in B_{\varepsilon_m}(0) \colon |A_m y| \geq
 \delta/2\}}
u(x_0+A_m y) dy \geq  \frac{C \delta^2}{\varepsilon_m^n}\,  \big|\big\{y\in B_{\varepsilon_m}(0) \colon |A_m y| \geq
 \delta/2\big\} \big|
$$
Next, since $
E_{\varepsilon_m}(A_m, x_0)\cap (\Omega \setminus B_{\delta}(x_0)) \neq \emptyset,
$
the largest eigenvalue of $A_m$ must satisfy
$$
\frac{\delta}{\varepsilon_m} \leq \lambda_{\max}=\lambda_{\max} (A_m).
$$
We assume that the eigenvector corresponding to $\lambda_{\max} (A_m)$ is in the direction of $x_n$, and we have
%(this is a control on the ellipsoid).
\[
\begin{split}
 \Big|\left\{ y\in B_{\varepsilon_m}(0) \colon |A_m y| \geq
 \delta/2\right\} \Big|
 &= \varepsilon^n_m \left| \left\{ x\in B_{1}(0) \colon  |A_m x| \geq
 \frac{\delta}{2\varepsilon_m} \right\} \right|\\
 &\geq 
 \varepsilon^n_m \left| \left\{ x\in B_{1}(0) \colon  (A_m x)_n \geq
 \frac{\delta}{2\varepsilon_m} \right\} \right|\\
 &\geq 
 \varepsilon^n_m \left| \left\{ x\in B_{1}(0) \colon  (A_m x)_n \geq
 \frac{\lambda_{\max}}{2} \right\} \right|,
\end{split}
\]
where $(A_m x)_n$ denotes the $n$-th coordinate of $A_m x$.
By Cavalieri's principle we have 
\[
\begin{split}
\left| \left\{ x\in B_{1}(0) \colon  (A_m x)_n \geq
 \frac{\lambda_{\max}}{2} \right\} \right|
&=\int_{\frac{\lambda_{\max}}{2}}^{\lambda_{\max}} |E_{1}(A_m, 0)\cap\{y:y_n=x_n\}|\,dx_n\\
&=|B_1^{(n-1)}(0)|\cdot\int_{\frac{\lambda_{\max}}{2}}^{\lambda_{\max}} \lambda_1\lambda_2\dots\lambda_{n-1} \left( 1-(x_n/\lambda_{\max})^2\right)^{\frac{n-1}{2}}\,dx_n\\
&=|B_1^{(n-1)}(0)|\cdot\int_{\frac{\lambda_{\max}}{2}}^{\lambda_{\max}} \frac{1}{\lambda_{\max}} \left( 1-(x_n/\lambda_{\max})^2\right)^{\frac{n-1}{2}}\,dx_n\\
&=|B_1^{(n-1)}(0)|\cdot\int_{\frac{1}{2}}^{1}\left( 1-y_n^2\right)^{\frac{n-1}{2}}\,dy_n\\
&=| B_1(0)\cap\{y: y_n\geq 1/2\}|,
\end{split}
\]
where we have used Remark \ref{properties.elipsoids} and the change of variable $x_n=\lambda_{\max} y_n$.
Therefore, we obtain
$$
\dashint_{B_{\varepsilon_m}(0)}
u(x_0+A_m y) \,dy \geq c(\delta) >0. $$
Using  this lower bound, \eqref{absolute.goal.C88.99} and Theorem \ref{theorem.C2.case} we conclude

\[
\begin{split}
o(\varepsilon_m^2)&=
\mathop{\mathop{\inf}_{\det A=1}}_
{A\leq \phi(\varepsilon_m)I}
\bigg\{
\dashint_{B_{\varepsilon_m}(0)}
\left(
u(x_0+Ay)
-
\frac{n}{2(n+2)}\,
\left(\det{D^2u(x_0)}\right)^{1/n}
\varepsilon_m^2
\right)
\,dy
\bigg\}
\\
&\geq\mathop{\mathop{\inf}_{\det A=1}}_
{E_{\varepsilon_m} (A,x_0) \subset\Omega}\bigg\{
\dashint_{B_{\varepsilon_m }(0)}
\left(
u(x_0+Ay)  -
\frac{n}{2(n+2)}\,
\left(\det{D^2u(x_0)}\right)^{1/n}
\varepsilon_m^2
\right)
\,dy
\bigg\}
\\
&\geq
\dashint_{B_{\varepsilon_m}(0)}
\left(
u(x_0+A_my)  -
\frac{n}{2(n+2)}\,
\left(\det{D^2u(x_0)}\right)^{1/n}
\varepsilon_m^2
\right)
\,dy
+ o(\varepsilon_m^2 )
\\
&\geq
c (\delta) 
+ o(\varepsilon_m^2 )
>0
\end{split}
\]
as $\varepsilon_m\to0$, a contradiction. This shows that a sequence $A_m$ verifying $\det A_m=1$, $E_{\varepsilon_m}(A_m, x_0)\subset \Omega$,  and 
$
E_{\varepsilon_m}(A_m, x_0) \cap (\Omega \setminus B_{\delta}) \neq \emptyset,
$
together with inequality \eqref{absolute.goal.C88.99} cannot exist. 
\end{proof}

\begin{proof}[Proof of Theorem \ref{theorem.C2.case2}]
We fix $x_0\in \Omega$ and, as in the previous lemma, assume that $u(x_0)=0$ and $u\geq 0$ in $\Omega$.  Since $u$ is convex we can subtract  a supporting hyperplane of $u$ at $x_0$. 
We split the proof in two cases: $\det(D^2u (x_0)) =0$ and $\lambda_{\min}(D^2u (x_0)) >0$.

In the first case, since  $\det(D^2u (x_0)) =0$  and $u(x_0)=0$, we have to prove that
\[
\mathop{\mathop{\inf}_{\det A=1}}_
{E_{\varepsilon} (A,x_0) \subset\Omega}
\dashint_{B_{\varepsilon }}
u(x_0+Ay) 
\,dy=
 o(\varepsilon^2).
\]
This follows from Theorem~\ref{theorem.C2.case} since
\[
0\leq
\mathop{\mathop{\inf}_{\det A=1}}_
{E_{\varepsilon} (A,x_0) \subset\Omega}
\dashint_{B_{\varepsilon }}
u(x_0+Ay) 
\,dy
\leq
\mathop{\mathop{\inf}_{\det A=1}}_
{A\leq \phi(\varepsilon)I}
\dashint_{B_{\varepsilon }}
u(x_0+Ay) 
\,dy
= o(\varepsilon^2)
\]
as $\varepsilon\to 0$.

In the second case, when the eigenvalues of $D^2 u(x_0)$ are strictly positive, Lemma \ref{keylemma} shows that we may assume that
the ellipsoids $E_\varepsilon(A,x_0)$ are contained in a ball $B_\delta(x_0)$ with $\delta$ as small as needed. 
We then argue as in the proof of Theorem \ref{theorem.C2.case}.\end{proof}

We conclude this section with the following remark.
\begin{remark}
\label{proof.border.formula}
Using the coarea formula, we can rewrite our mean value formulas in terms of a weighted surface integral over  $\partial E_{\varepsilon}(A,x)$  as follows,
\begin{equation}\label{MVF.surface.weighted}
\begin{split}
u(x)
=
\mathop{\mathop{\inf}_{
\det A=1,}}_{
A\leq\phi(\varepsilon) I}\left\{
\frac{
\displaystyle\int_{\partial E_\varepsilon(A,x)}
u(y)\,|(AA^t)^{-1}(y-x)|^{-1}\,d\mathcal{H}^{n-1}(y)
}
{\displaystyle
\int_{\partial E_\varepsilon (A,x)}
|(AA^t)^{-1}(y-x)|^{-1}\,d\mathcal{H}^{n-1}(y)
}
\right\}
-\frac{1}{2}\,(f(x))^{1/n}\,{ \varepsilon}^2+o(\varepsilon^2).
\end{split}
\end{equation}
To see this, notice that for a paraboloid $P$, a change of variables in \eqref{proof.MVT.intermediate.step.1}  yields
\[
\dashint_{E_{\varepsilon}(A,x)}
\left(
 P(y)- P(x)
\right)
\,dy
=
\frac{\varepsilon^2}{2(n+2)}
\textnormal{trace}\big(A^tD^2 P(x)A\big).
\]
For any $g\in L^1$, the coarea formula
gives
\begin{equation}\label{lemita.coarea.remark}
\int_{E_\varepsilon(A,x)}
g(y)\,dy
=
\int_0^\varepsilon
t
\int_{\partial E_t(A,x)}
g(y)\,|(AA^t)^{-1}(y-x)|^{-1}\,d\mathcal{H}^{n-1}(y)
\,dt.
\end{equation}
In particular, for $g(y)= P(y)- P(x)$ we have
\[
\begin{split}
\int_0^\varepsilon
&t
\int_{\partial E_t(A,x)}
\left(
 P(y)- P(x)
\right)\,|(AA^t)^{-1}(y-x)|^{-1}\,d\mathcal{H}^{n-1}(y)
\,dt
\\
&=\int_{E_\varepsilon(A,x)}
\left(
 P(y)- P(x)
\right)
\,dy
=
\frac{|B_1(0)|}{2(n+2)}\,
\textnormal{trace}\big(A^tD^2 P(x)A\big)
\varepsilon^{n+2}.
\end{split}
\]
Then, we can differentiate both sides with respect to $\varepsilon$ and obtain
\begin{equation}\label{lemita.coarea.remark1}
\begin{split}
\int_{\partial E_\varepsilon(A,x)}
\left(
 P(y)- P(x)
\right)\,|(AA^t)^{-1}(y-x)|^{-1}\, d\mathcal{H}^{n-1}(y)
=\frac{|B_1(0)|}{2}\,
\textnormal{trace}\big(A^tD^2 P(x)A\big)
\varepsilon^{n}.
\end{split}
\end{equation}
Observe that making
$g(y)=1$ in
\eqref{lemita.coarea.remark} and differentiating both sides with respect to $\varepsilon$ gives
\begin{equation}\label{lemita.coarea.remark2}
n\varepsilon^{n-2}\,|B_1(0)|
=
\int_{\partial E_\varepsilon(A,x)}
|(AA^t)^{-1}(y-x)|^{-1}\,d\mathcal{H}^{n-1}(y).
\end{equation}
Then, \eqref{lemita.coarea.remark1} and \eqref{lemita.coarea.remark2}
 lead to 
\[
\frac{
\displaystyle\int_{\partial E_\varepsilon(A,x)}
 P(y)\,|(AA^t)^{-1}(y-x)|^{-1}\,d\mathcal{H}^{n-1}(y)
}
{\displaystyle
\int_{\partial E_\varepsilon(A,x)}
|(AA^t)^{-1}(y-x)|^{-1}\,d\mathcal{H}^{n-1}(y)
}
=
 P(x)
+
\frac{\varepsilon^2}{2n}\,
\textnormal{trace}\big(A^tD^2 P(x)A\big),
\] 
 and \eqref{MVF.surface.weighted} follows in the same way as in the proof of Theorems \ref{theorem.C2.case} and \ref{thm.MVP.visco}.
\end{remark}

\section{Examples}\label{sect.examples}

In this section we present  several examples referred to in the introduction that  highlight different aspects of Theorems \ref{theorem.C2.case}, \ref{thm.MVP.visco},  \ref{theorem.C2.case2}, and
\ref{thm.MVP.visco2}. 
\par
In the first two examples, we illustrate Theorem \ref{theorem.C2.case} for convex paraboloids, with emphasis in the case when $\lambda=0$ is an eigenvalue of the Hessian matrix of the paraboloid. 
As pointed out in the introduction, 
strictly convex paraboloids satisfy \eqref{absolute.goal.C2} without remainder.
However, when $\lambda = 0$ is an eigenvalue of the Hessian matrix of the paraboloid, the remainder $o(\varepsilon^2)$ appears due to the 
restriction $A\leq\phi( \varepsilon)I$. 
\begin{example}\label{equaltyforallmatrices}
For all convex paraboloids we have equality without the remainder in the expansion \eqref{absolute.goal.C2.2} considering all matrices satisfying $\det(A)=1$.
\end{example}
 By definition, a paraboloid $P$ coincides with its second-order Taylor expansion around $x$ and we can write
\[
\begin{split}
\dashint_{{B_\varepsilon(0)}}
P(x+Ay)
\,dy
&=
\dashint_{{B_\varepsilon(0)}}
\Big(
P(x)
+
\langle A^t\nabla P(x),y\rangle
+
\frac12\langle A^tD^2P(x)Ay,y\rangle 
\Big)
\,dy
\\
&=
P(x)
+
\frac12
\dashint_{{B_\varepsilon}}
\langle A^tD^2P(x)Ay,y\rangle 
\,dy,
\end{split}
\]
since the integral of the first-order term vanishes by symmetry. Then,
Lemma \ref{lemma.trace.integral} gives
\begin{equation}\label{proof.MVT.intermediate.step.1}
\dashint_{{B_\varepsilon}}
P(x+Ay)
\,dy
=P(x)
+\frac{\varepsilon^2}{2(n+2)}
\,
{\rm trace}\big(A^tD^2P(x)A\big).
\end{equation}
If the paraboloid $P$ is convex, we can take infimum over the class of matrices with determinant equal to 1 on both sides of \eqref{proof.MVT.intermediate.step.1} and apply Lemma \ref{caract.determ}, to obtain 
\begin{equation}\label{proof.MVT.step1}
\inf_{\det A=1}\dashint_{{B_\varepsilon}}
P(x+Ay)
\,dy
=P(x)
+\frac{n}{2(n+2)}
\,
\big(\det D^2P(x)\big)^{1/n}\,\varepsilon^2.
\end{equation}

\begin{example}\label{example.paraboloids.what.happens.with.zero.eigs}
For strictly convex paraboloids we have equality without remainder in the expansion \eqref{absolute.goal.C2} restricted to matrices $A$  satisfying
$A\le \phi(\varepsilon) I$ and $\det(A)=1$. The remainder is present when the paraboloid is convex but not strictly convex (and not identically zero).
\end{example}
%Notice that this can be done for any convex paraboloid $P$, whether strictly convex or not.

Taking infima on both sides of \eqref{proof.MVT.intermediate.step.1} with the additional constraint  $A\leq \phi(\varepsilon)I$ as in Theorem \ref{theorem.C2.case} gives
\begin{equation}\label{example.paraboloid.zero.eig.additional.constraint}
\begin{split}
\mathop{\mathop{\inf}_{\det A=1,}}_{
A\leq \phi(\varepsilon) I} \bigg\{
\dashint_{B_{\varepsilon}}
P(x+Ay)
\,dy
\bigg\}
=
P(x)+
\frac{\varepsilon^2}{2(n+2)}
\mathop{\mathop{\inf}_{\det A=1,}}_{
A\leq \phi(\varepsilon)I} \bigg\{
{\rm trace}\big(A^tD^2P(x)A\big)
\bigg\}.
\end{split}
\end{equation}
If $D^2P(x)>0$, we can apply Lemmas \ref{lemma.uniform.ellipticity.visco} and \ref{caract.determ} and get \eqref{absolute.goal.C2} 
without a remainder when $\varepsilon$ is small enough.
Let us now consider the case  when $\lambda =0$ 
is an eigenvalue of the matrix  $D^2P(x)$ with multiplicity $n-k<n$ (assume that $P\not\equiv 0$,
otherwise everything trivializes). Suppose that we have equality without remainder  in  \eqref{absolute.goal.C2}  for small $\varepsilon$.
We would have
 \begin{equation}\label{equality}
P(x)=
\mathop{\mathop{\inf}_{\det A=1,}}_{
A\leq \phi(\varepsilon)I}\bigg\{
\dashint_{B_{\varepsilon}}
P(x+Ay)
\,dy \bigg\},
\end{equation}
which, in view of \eqref{example.paraboloid.zero.eig.additional.constraint}, implies that for $\varepsilon>0$
$$
\mathop{\mathop{\inf}_{\det A=1,}}_{
A\leq \phi(\varepsilon)I} \bigg\{
{\rm trace}\big(A^tD^2P(x)A\big)
\bigg\}=0.
$$
We show that this is not possible. Write
\[
D^2P(x)=Q
\,{\rm diag}
\big(
\lambda_1(D^2P(x)),\ldots,\lambda_k(D^2P(x)),0,\ldots,0
\big)\,
Q^t
\]
with $Q Q^t=I$. 
For $\varepsilon>0$ small consider matrices $A$ such that
$\det(A)=1$ and $A\leq \phi(\varepsilon)I$. We have
\[
{\rm trace}( A^tD^2P(x)A)\geq \frac{1}{\phi(\varepsilon)^{2(n-1)}}
\min_{j=1,..,k}\lambda_j(D^2P(x)) >0
\]
independently of $A$, no matter how small $\varepsilon$ is (this is due to the fact that $\min_{j=1,..,k} \lambda_j (D^2 P(x))>0$).

The following example shows that for a general smooth function (without 
convexity), we have
\begin{equation}\label{examples.infima.are.different}
\inf_{\det A=1}
\dashint_{B_{\varepsilon}(0)}
u(x+Ay)\,dy
\neq
\mathop{\mathop{\inf}_{\det A=1,}}_{
A\leq\phi(\varepsilon) I} \bigg\{
\dashint_{B_{\varepsilon}(0)}
u(x+Ay)\,dy
\bigg\}
+o(\varepsilon^2).
\end{equation}
The problem is the change in convexity outside of the domain, which is detected  by ``far reaching'' matrices $A$. 
This is not seen when we require $A\leq \phi(\varepsilon)$ because  we remain inside the domain for $\varepsilon$ small enough.

\begin{example} \label{example.fourth.order.negative}
Consider the functions
\[
u^+(x)=\frac{|x|^2}{2} + \frac{|x|^4}{12}\qquad
\textrm{and}\qquad u^-(x)=\frac{|x|^2}{2} - \frac{|x|^4}{12}.
\]
Then, $u^+,u^-$ respectively solve $\det{D^2u(x)}=f(x)$  in the classical sense for
\[
f^+(x)=(1+|x|^2)\Big(1+\frac{|x|^2}{3}\Big)^{n-1}
\qquad\textrm{and}\qquad
f^-(x)=(1-|x|^2)\Big(1-\frac{|x|^2}{3}\Big)^{n-1}.
\]
Moreover,  for $\phi(\varepsilon)$ under hypotheses \eqref{hipotesis.phi}, we have
\begin{equation} \label{examples.infima.are.different.999}
\inf_{\det A=1}
\dashint_{B_{\varepsilon}(0)}
u^-(x+Ay)\,dy
=-\infty, 
\end{equation}
and 
\begin{equation}\label{examples.infima.are.different.99}
\begin{split}
\mathop{\mathop{\inf}_{\det A=1}}_
{A\leq\phi(\varepsilon) I}
\dashint_{B_{\varepsilon}(0)}
u^-(x+Ay)\,dy=
\mathop{\mathop{\inf}_{\det A=1}}_
{E_\varepsilon (A,x) \subset\Omega}
\dashint_{B_{\varepsilon}(0)}
u^-(x+Ay)
\,dy +o(\varepsilon^2)
\end{split}
\end{equation}
where  
we take $\Omega=B_{1}(0)$ so that 
$u^-$ is convex and $f^-(x)\geq 0$ in $\Omega$. This is in contrast with $u^+$, for which
\begin{equation}\label{examples.infima.are.all.equal}
\begin{split}
\inf_{\det A=1}
\dashint_{B_{\varepsilon}(0)}
u^+(x+Ay)\,dy
=
\mathop{\mathop{\inf}_{\det A=1}}_
{A\leq\phi(\varepsilon) I}
\dashint_{B_{\varepsilon}(0)}
u^+(x+Ay)\,dy +o(\varepsilon^2)
\end{split}
\end{equation}
as a consequence of  Theorems \ref{theorem.C2.case2} and \ref{theorem.C2.case}.
\end{example}

Let us show first with an explicit calculation that
\[
\mathop{\mathop{\inf}_{\det A=1}}_
{A\leq  \phi(\varepsilon)I}\bigg\{
\dashint_{B_{\varepsilon}(0)}
\big(u^{-}(x+Ay)-u^{-}(x)\big)\,dy
\bigg\}
=
\frac{n}{2(n+2)}
\,(f(x))^{1/n}\varepsilon^2
+o(\varepsilon^2).
\]
To see this, first,  observe that the matrix 
\begin{equation}\label{example.fourth.order.Hessian}
D^2u^{-}(x)
=
\Big(
1-\frac{|x|^2}{3}
\Big)\,
I
-
\frac23
\big(
x \otimes x
\big)
\end{equation}
has two eigenvalues, $1-|x|^2$, which is simple and corresponds to the eigenvector $x$, and $1-|x|^2/3$, which has multiplicity $n-1$ and eigenspace  given by the subspace orthogonal to $x$. 
 Then,
\begin{equation}\label{example.fourth.order.minus.0}
 \begin{split}
\dashint_{B_{\varepsilon}(0)}
\big(u^{-}(x+Ay)-u^{-}(x)\big)\,dy
=
\dashint_{B_{\varepsilon}(0)}
\left(
\frac{|x+Ay|^2}{2}-\frac{|x+Ay|^4}{12}
-
\frac{|x|^2}{2}+\frac{|x|^4}{12}
\right)\,dy&
\\
=
\frac{1}{2}
\Big(
1-\frac{|x|^2}{3}
\Big)
\dashint_{B_{\varepsilon}(0)}
|Ay|^2
\,dy
-
\frac{1}{12}
\dashint_{B_{\varepsilon}(0)}
|Ay|^4
\,dy
-
\frac{1}{3}
\dashint_{B_{\varepsilon}(0)}
\langle x, Ay\rangle^2
\,dy&
\end{split}
\end{equation}
since $y\mapsto\langle x,Ay\rangle$ is  odd and the integral over $B_{\varepsilon}(0)$ vanishes. 
We compute the three integrals on the right-hand side. 
First, Lemma \ref{lemma.trace.integral} yields
 \begin{equation}\label{example.fourth.order.plus.01}
\dashint_{{B_\varepsilon(0)}}
|Ay|^2
\,dy
=
\frac{\varepsilon^2}{n+2}
\,
{\rm trace}\big(A^tA\big).
\end{equation}
For the second integral, a change to polar coordinates gives
 \begin{equation}\label{example.fourth.order.plus.02}
\begin{split}
\dashint_{B_{\varepsilon}(0)}
|Ay|^4
\,dy
=
\frac{\varepsilon^{4}}{(n+4)|B_1(0)|}
\int_{\partial B_{1}(0)}
|Ay|^4
\,d\mathcal{H}^{n-1}(y).
\end{split}
\end{equation}
Notice that the exterior unit normal to $\partial B_1(0)$ at  $y\in \partial B_1(0)$ is $y$ itself. Therefore, the divergence theorem yields
 \begin{equation}\label{example.fourth.order.plus.03}
\begin{split}
\int_{\partial B_{1}(0)}
|Ay|^4
\,d\mathcal{H}^{n-1}(y)
&=
\int_{ B_{1}(0)}
\textnormal{div}\left(
|Ay|^2A^tAy
\right)
\,dy
\\
&=
2\int_{ B_{1}(0)}
|A^tAy|^2
\,dy
+
\textnormal{trace}(A^tA)
\int_{ B_{1}(0)}
|Ay|^2
\,dy
\\
&=
\frac{|B_1(0)|}{n+2}
\left(
2\,\textnormal{trace}(A^tAA^tA)
+
\left(\textnormal{trace}(A^tA)\right)^2
\right),
\end{split}
\end{equation}
where we have applied  Lemma \ref{lemma.trace.integral} twice in the last step.
Finally, a similar argument, changing to polar coordinates and applying the divergence theorem yields
 \begin{equation}\label{example.fourth.order.plus.04}
\begin{split}
\dashint_{B_{\varepsilon}(0)}
\langle x, Ay\rangle^2
\,dy&
=
\frac{\varepsilon^2}{(n+2)|B_1(0)|}
\int_{\partial B_{1}(0)}
\langle x, Ay\rangle^2
\,d\mathcal{H}^{n-1}(y)
\\
&=
\frac{\varepsilon^2}{(n+2)|B_1(0)|}
\int_{ B_{1}(0)}
\textnormal{div}_y\big(
\langle A^tx, y\rangle 
A^tx
\big)
\,dy
=
\frac{\varepsilon^2}{n+2}
|A^tx|^2.
\end{split}
\end{equation}
Plugging \eqref{example.fourth.order.plus.01}--\eqref{example.fourth.order.plus.04}  back into \eqref{example.fourth.order.minus.0} gives
\begin{equation}\label{example.fourth.order.identity}
 \begin{split}
 &\frac{2(n+2)}{\varepsilon^2}
\dashint_{B_{\varepsilon}(0)}
\big(u^{-}(x+Ay)-u^{-}(x)\big)\,dy
\\
&\quad=
\Big(
1-\frac{|x|^2}{3}
\Big)
\,
{\rm trace}\big(A^tA\big)
-\frac{\varepsilon^{2}}{6(n+4)}
\left(
2\,\textnormal{trace}(A^tAA^tA)
+
\left(\textnormal{trace}(A^tA)\right)^2
\right)
-
\frac{2|A^tx|^2}{3}
\\
&\quad=
{\rm trace}\big(A^tD^2u^{-}(x)A\big)
-\frac{\varepsilon^{2}}{6(n+4)}
\left(
2\,\textnormal{trace}(A^tAA^tA)
+
\left(\textnormal{trace}(A^tA)\right)^2
\right),
\end{split}
\end{equation}
with $D^2u^{-}(x)$ given by \eqref{example.fourth.order.Hessian}.

Let us see what happens if we impose $A\leq \phi(\varepsilon)I$ with $\phi(\varepsilon)$ under hypotheses \eqref{hipotesis.phi}. We get
\[
\begin{split}
{\rm trace}&\left(A^t\left(D^2u^{-}(x)
-\frac{n+2}{6(n+4)}
\varepsilon^{2}\phi(\varepsilon)^2 I
\right)A\right)
\\
&\leq
 \frac{2(n+2)}{\varepsilon^2}
\dashint_{B_{\varepsilon}(0)}
\big(u^{-}(x+Ay)-u^{-}(x)\big)\,dy
\leq
{\rm trace}\big(A^tD^2u^{-}(x)A\big)
\end{split}
\]
and, from there, by \eqref{hipotesis.phi}, we get
\[
\begin{split}
\frac{\varepsilon^2}{2(n+2)}
&\mathop{\mathop{\inf}_{\det A=1}}_
{A\leq  \phi(\varepsilon)I}
 \Big\{
{\rm trace}\Big(A^t\big(D^2u^{-}(x)
-o(1)I
\big)A\Big)
\Big\}
\\
&\leq
\mathop{\mathop{\inf}_{\det A=1}}_
{A\leq  \phi(\varepsilon)I} \bigg\{
\dashint_{B_{\varepsilon}(0)}
\big(u^{-}(x+Ay)-u^{-}(x)\big)\,dy
\bigg\}
\leq
\frac{\varepsilon^2}{2(n+2)}
\mathop{\mathop{\inf}_{\det A=1}}_
{A\leq  \phi(\varepsilon)I}\Big\{
{\rm trace}\big(A^tD^2u^{-}(x)A\big)
\Big\}
\end{split}
\]
as $\varepsilon\to0$.
At this point, we can apply Lemmas \ref{lemma.uniform.ellipticity.visco} and \ref{caract.determ} (notice that $D^2u^{-}(x)>0$ in $\Omega$) and get that, for $\varepsilon$ small enough
\[
\begin{split}
\frac{n}{2(n+2)}
&
\Big(\det\big(D^2u^{-}(x)
-o(1)I
\big)\Big)^{1/n}
\varepsilon^2
\\
&\leq
\mathop{\mathop{\inf}_{\det A=1}}_
{A\leq  \phi(\varepsilon)I}
 \bigg\{
\dashint_{B_{\varepsilon}(0)}
\big(u^{-}(x+Ay)-u^{-}(x)\big)\,dy
\bigg\}
\leq
\frac{n}{2(n+2)}
\big(\det{D^2u^{-}(x)}\big)^{1/n} \varepsilon^2.
\end{split}
\]
We conclude using
\eqref{C2.case.determinant.estimates} to get
\[
\begin{split}
\frac{n}{2(n+2)}
\,(f(x))^{1/n}\varepsilon^2
-o(\varepsilon^2)
\leq
\mathop{\mathop{\inf}_{\det A=1}}_
{A\leq  \phi(\varepsilon)I}\bigg\{
\dashint_{B_{\varepsilon}(0)}
\big(u^{-}(x+Ay)-u^{-}(x)\big)\,dy
\bigg\}
\leq
\frac{n}{2(n+2)}
\,(f(x))^{1/n}\varepsilon^2.
\end{split}
\]

Furthermore, we know by Theorem \ref{theorem.C2.case2} that 
\[
\mathop{\mathop{\inf}_{\det A=1}}_
{A\leq\phi(\varepsilon)I }
\dashint_{B_{\varepsilon}(0)}
u^{-}(x+Ay)\,dy
=\mathop{\mathop{\inf}_{\det A=1}}_
{E_\varepsilon (A,x) \subset\Omega}
\dashint_{B_{\varepsilon}(0)}
u^{-}(x+Ay)
\,dy +o(\varepsilon^2), 
\]
although this fact is not obvious from \eqref{example.fourth.order.identity}.
 
 A similar computation to the above
shows that the function
 $$
 u^+(x)= \frac{|x|^2}{2}+ \frac{ |x|^4}{12}
 $$ 
satisfies
\[
 \begin{split}
 n\big(\det{D^2u^{+}(x)}\big)^{1/n}
 &=
 \inf_{\det{A}=1}
 {\rm trace}\big(A^tD^2u^{+}(x)A\big)
 \\
 &\leq
 \frac{2(n+2)}{\varepsilon^2}
  \inf_{\det{A}=1}
\dashint_{B_{\varepsilon}(0)}
\big(u^{+}(x+Ay)-u^{+}(x)\big)\,dy
\\
&\leq
{\rm trace}\big(A^tD^2u^{+}(x)A\big)
+\frac{\varepsilon^{2}}{6(n+4)}
\left(
2\,\textnormal{trace}(A^tAA^tA)
+
\left(\textnormal{trace}(A^tA)\right)^2
\right),
\end{split}
\]
for every $A$ with $\det{A}=1$. In particular, we can take $A=\det(D^2u^{+}(x))^{1/(2n)}D^2u^{+}(x)^{-1/2}$ and get
\[
\begin{split}
u^{+}(x)
&=
\inf_{\det A=1}
\dashint_{B_{\varepsilon}(0)}
u^{+}(x+Ay)\,dy
-\frac{n}{2(n+2)}\,(f(x))^{1/n}\,\varepsilon^2
+O(\varepsilon^4)
\end{split}
\]
as $\varepsilon \to0$, from which \eqref{examples.infima.are.all.equal} follows

To conclude the example let us see that
\begin{equation}\label{example.negative.inf.unbounded}
\inf_{\det A=1}
\dashint_{B_{\varepsilon}(0)}
\big(u^{-}(x+Ay)-u^{-}(x)\big)\,dy
=
-\infty,
\end{equation}
and  \eqref{examples.infima.are.different.999} holds. This shows that the ellipsoids in the mean value property need to be restricted to the domain where $u$ is convex.
To prove \eqref{example.negative.inf.unbounded}, notice that  \eqref{example.fourth.order.identity}, yields
\[
 \begin{split}
 \frac{2(n+2)}{\varepsilon^2}&
\inf_{\det A=1}\dashint_{B_{\varepsilon}(0)}
\big(u^{-}(x+Ay)-u^{-}(x)\big)\,dy
\\
&=
\inf_{\det A=1}\left\{
{\rm trace}\big(A^tD^2u^{-}(x)A\big)
-\frac{\varepsilon^{2}}{6(n+4)}
\left(
2\,\textnormal{trace}(A^tAA^tA)
+
\left(\textnormal{trace}(A^tA)\right)^2
\right)
\right\}
\end{split}
\]
with $D^2u^{-}(x)$ given by \eqref{example.fourth.order.Hessian}. Let us assume that the matrix $A$ is diagonal. Then,
\[
 \begin{split}
{\rm trace}\big(A^tD^2u^{-}(x)A\big)
&-\frac{\varepsilon^{2}}{6(n+4)}
\left(
2\,\textnormal{trace}(A^tAA^tA)
+
\left(\textnormal{trace}(A^tA)\right)^2
\right)
\\
&=
\sum_{i=1}^{n}
\lambda_i^2(A)
\,u_{x_ix_i}(x)
-
\frac{\varepsilon^{2}}{6(n+4)}
\left(
2\sum_{i=1}^{n}
\lambda_i^4(A)
+
\left(
\sum_{i=1}^{n}
\lambda_i^2(A)
\right)^2
\right)
\end{split}
\]
and the goal is to minimize this expression with the restriction $\det(A)=1$.
For simplicity, let us further assume $n=2$ here, since the general case is similar. 
In two dimensions, $\det(A)=1$ means that the two eigenvalues of $A$ are necessarily reciprocal 
and we can write $\lambda_1(A)=(\lambda_2(A))^{-1}=\lambda$. Then,  we have
\[
 \begin{split}
{\rm trace}\big(A^tD^2u^{-}(x)A\big)
-\frac{\varepsilon^{2}}{36}
\left(
2\,\textnormal{trace}(A^tAA^tA)
+
\left(\textnormal{trace}(A^tA)\right)^2
\right)&
\\
=
\lambda^2
\,u_{x_1x_1}(x)
+
\frac{u_{x_2x_2}(x)}{\lambda^2}
-
\frac{\varepsilon^{2}}{12}
\left(
\lambda^4
+\frac{1}{\lambda^4}
\right)
-
\frac{\varepsilon^{2}}{18}.&
\end{split}
\]
Since we are assuming no relation between $A$ and $\varepsilon$, we can fix $\varepsilon$ and let $\lambda\to0$, 
which gives~\eqref{example.negative.inf.unbounded}.

 In the following example we present a viscosity solution to the 
 Monge-Amp\`ere equation that is not classical. Nevertheless, the asymptotic mean value formula holds in the point-wise sense.

\begin{example} \label{example4.4} 
Consider the Monge-Amp\`ere equation $\det{D^2u(x)}=f(x)$ with 
\[ 
u(x) = \frac{1}{2}(|x|-1)_+^2, 
\qquad
f(x) = \left(1-\frac{1}{|x|}\right)_+^{n-1}.
\]

This is an example of a viscosity solution
of the Monge-Amp\`ere equation with right-hand side $f(x)$ 
 that is not  a classical $C^2$ solution.    Notice that $f=0$ in 
$\overline{B}_1(0)$.
We show
 how much simpler it is to work with the asymptotic mean value property in the viscosity sense, in contrast with 
 in the point-wise sense, where we have to compute an infimum of integrals for the actual function $u$ (and not just for the test
 paraboloid). In this example, the resulting integrals for $u$ are of elliptic type. Nevertheless, we are able to use the viscosity 
 computation to show that the asymptotic mean value formula holds in the point-wise sense.  
 \end{example}

For the reader's convenience, we will  verify the definition of viscosity solution. This only needs to be done in the sphere $\partial B_1(0)$ since $u$ is locally $C^2$ otherwise. The subsolution case is immediate since any convex paraboloid $P$ that touches $u$ from above at $x_0$  will automatically satisfy $\det D^2P(x_0) \geq 0=f(x_0)$ by convexity. 
Let us check now the supersolution case. Let $P$ be a  convex paraboloid that  touches $u$ from below at $x_0\in\partial B_1(0)$, i.e.,
 such that $P \leq u$ with $P(x_0) = u(x_0) = 0$.  
Since $\nabla u(x_0) = 0$, we must require $\nabla P(x_0) = 0$ as well.  Moreover, $u$ is constant and equal to zero in part of any neighborhood of $x_0$ and, by convexity and the contact condition, we find that $P(x)=0$ is the only convex paraboloid that touches $u$ from below at $x_0$ and $\det D^2 P(x_0) = 0$, as required by the definition of  viscosity supersolution.

Assume now that  $P(x)$ is a convex paraboloid that touches $u$ from above at $x_0$, with $|x_0|=1$.
Because $u(x_0)=0$ and $\nabla u(x_0)=0$,
then, $P(x)$ must be of the form
\begin{equation} \label{example.visc.defn.para}
P(x)
=
\frac12
\big\langle
M(x-x_0),(x-x_0)
\big\rangle
\end{equation}
Moreover, in a neighborhood of $x_0$ we must have
\begin{equation} \label{example.visc.ineq.we.want}
 \frac{1}{2}(|x|-1)_+^2
= u(x)
 \leq
 P(x)=
\frac12
\big\langle
M(x-x_0),(x-x_0)
\big\rangle,
\end{equation}
which holds trivially for $|x|\leq1$.
Taking $x=(1+a)x_0$ in the previous equation we obtain 
that
\begin{equation}\label{example.visc.necessary.condition}
\langle
Mx_0,x_0
\rangle
\geq1
\end{equation}
is a  necessary condition for contact. For example, if $M\geq I$, then 
\[
2P(x)=
\langle
M(x-x_0), (x-x_0)
\rangle
\geq
|x-x_0|^2
=
|x|^2
-
2
\langle
x,x_0
\rangle
+1
\geq
|x|^2
-
2
|x|
+1
=
2u(x)
\]
 and this holds for all $x$.
However, \eqref{example.visc.necessary.condition} is not a sufficient condition. As we are going to see, one can take some of the eigenvalues of $M$ smaller than 1 in the directions orthogonal to $x_0$ and still have that $P$ touches $u$ from above in a small neighborhood around $x_0$. 

Let us choose in \eqref{example.visc.defn.para} the matrix
\begin{equation}\label{example.visc.defn.para.M}
M
=
Q\,\textnormal{diag}(\lambda_1,\ldots,\lambda_n)Q^t,
\end{equation}
with $Q$ an orthogonal matrix with first column $x_0$ and the remaining $n-1$ forming an orthonormal basis of the orthogonal complement to $x_0$. Furthermore, we will choose
\[
\lambda_1=1\qquad\textrm{and}\qquad\lambda_i=\lambda>0, \quad  i\neq1.
\]
Then, we are going to show that $P$ touches $u$ from above at $x_0$ in the ball $B_\delta (x_0)$ with
\begin{equation}
 \label{example.visc.defn.para.epsilon}
\delta = \frac{\lambda}{1-\lambda}.
\end{equation}
Notice that $\delta \to\infty$ as $\lambda\to1$ (which means that the neighborhood becomes larger and larger as we approach $M=I$)
 and $\delta \to0$ as $\lambda\to0$ (the neighborhood degenerates as the paraboloids become flatter in the directions orthogonal to $x_0$).

Denote by $w_i$ the columns of $Q$. Because they form an orthonormal basis of $\mathbb{R}^n$, we can write
\[
x
=
x_0
+
\sum_{i=1}^{n}
c_iw_i
=
(1+c_1)x_0
+
\sum_{i=2}^{n}
c_iw_i
\]
and
\begin{equation} \label{example.visc.pitagoras}
 |x|^2=(1+c_1)^2+\sum_{i=2}^{n}c_i^2.
\end{equation}
Therefore,
\[
\big\langle
M(x-x_0),(x-x_0)
\big\rangle
=
c_1^2+\lambda\sum_{i=2}^{n}c_i^2
=
c_1^2
-\lambda(1+c_1)^2
+
\lambda|x|^2,
\]
and proving \eqref{example.visc.ineq.we.want} for $|x|>1$ becomes equivalent to showing
\begin{equation} \label{example.visc.ineq.we.want2}
(|x|-1)^2
-
\lambda|x|^2
\leq
c_1^2
-\lambda(1+c_1)^2.
\end{equation}
Observe that when $\lambda\leq1$ the function $g(t)=t^2-\lambda(t+1)^2$ is decreasing for $t\leq\lambda/(1-\lambda)$. Moreover, \eqref{example.visc.pitagoras} and \eqref{example.visc.defn.para.epsilon} imply that 
\[
c_1
\leq
|x|-1
\leq
|x-x_0|+|x_0|-1\leq\frac{\lambda}{1-\lambda}
\]
and therefore
$g(|x|-1)\leq g(c_1)$,
which is equivalent to \eqref{example.visc.ineq.we.want2} and \eqref{example.visc.ineq.we.want}.

Let us show now that, in this case,
\begin{equation}\label{example.viscos.point-wise}
\begin{split}
\inf_{\det A=1}
\dashint_{B_{\varepsilon}(0)}
u(x_0+Ay)\,dy
=
\mathop{\mathop{\inf}_{\det A=1,}}_{
A\leq \phi(\varepsilon)I} \bigg\{
\dashint_{B_{\varepsilon}(0)}
u(x_0+Ay)\,dy
\bigg\}=u(x_0)+o(\varepsilon^2)
\end{split}
\end{equation}
as $\varepsilon \to 0$, for $\phi$ satisfying \eqref{hipotesis.phi}.
First, notice that for every $P$ that touches $u$ from above at $x_0$ in a neighborhood, which must be of the from \eqref{example.visc.defn.para}, we have
\begin{equation}\label{example.visc.chain.point-wise}
\begin{split}
0&\leq\inf_{\det A=1}
\dashint_{B_{\varepsilon}(0)}
u(x_0+Ay)\,dy
\leq
\mathop{\mathop{\inf}_{\det A=1,}}_{
A\leq \phi(\varepsilon)I} \bigg\{
\dashint_{B_{\varepsilon}(0)}
u(x_0+Ay)\,dy
\bigg\}
\\
&\leq
\mathop{\mathop{\inf}_{\det A=1,}}_{
A\leq \phi(\varepsilon)I}\bigg\{
\dashint_{B_{\varepsilon}(0)}
P(x_0+Ay)\,dy
\bigg\}
=
\frac{\varepsilon^2}{2(n+2)}\,
\mathop{\mathop{\inf}_{\det A=1,}}_{
A\leq \phi(\varepsilon)I} \Big\{
{\rm trace}(A^tMA)
\Big\}
\end{split}
\end{equation}
for $\varepsilon$ small enough, where in the last equality we have applied Lemma \ref{lemma.trace.integral}.

In particular, when we choose $M$ given by
\eqref{example.visc.defn.para.M}, the paraboloid $P$ touches $u$ from above at $x_0$ in a ball $B_\delta (x_0)$ with $\delta$ given by
\eqref{example.visc.defn.para.epsilon}. Moreover, we can apply
 Lemma \ref{lemma.uniform.ellipticity.visco} as long as
 \[
\left(
\frac{1+(n-1)\lambda}{\lambda}
\right)^{1/2} 
<
\phi(\varepsilon)
\qquad\textrm{and}\qquad
 \varepsilon <\frac{\lambda}{1-\lambda}=\delta ,
 \]
% \[
% \varepsilon <\frac{\lambda}{1+(n-1)\lambda}\leq\delta ,
% \]
which holds for $\varepsilon$ small enough.
This, together with \eqref{example.visc.chain.point-wise}, yields
\[
\begin{split}
0&\leq\inf_{\det A=1}
\dashint_{B_{\varepsilon}(0)}
u(x_0+Ay)\,dy \leq
\mathop{\mathop{\inf}_{\det A=1,}}_{
A\leq \phi(\varepsilon)I} \bigg\{
\dashint_{B_{\varepsilon}(0)}
u(x_0+Ay)\,dy
\bigg\}
\\
&\leq
\frac{n}{2(n+2)}\,
\left(\det{M}\right)^{1/n}\varepsilon^2
\leq
C\lambda^{1-\frac{1}{n}}
\varepsilon^2
\end{split}
\]
with $C$ independent of $\lambda$. Since $\lambda$ can be taken arbitrarily small, 
we conclude \eqref{example.viscos.point-wise}.

\begin{example} \label{Ex-MVP-viscosa}
This time we consider  the function $$u(x_1,x_2)=|x_1|$$ in $\mathbb{R}^2$,
which  is a viscosity solution to the Monge-Amp\`ere equation with $f=0$.
The mean value formula \eqref{MVP.solid.visco} does not hold point-wise for any $\phi$ satisfying \eqref{hipotesis.phi}.
The mean value formula \eqref{MVP.solid.visco2} holds in a point-wise sense if we consider the function in  its natural domain $\mathbb R^2$ (see Example~\ref{ejemplopablo} for a function where it does not hold point-wise).
On the other hand, if we restrict the domain to $B_1(0)$, the expansion \eqref{MVP.solid.visco2} does not hold in a point-wise sense.

More generally, we consider $$u(x_1,x_2)=|x_1|^\gamma,$$ with $\gamma \geq 1$ in $\mathbb{R}^2$.
We have that
$u(x_1, x_2)$ is a viscosity solution  for the Monge-Amp\`ere equation with 
$f=0$  and the mean value property \eqref{MVP.solid.visco}  holds point-wise  only when $\phi(\varepsilon)$ satisfies 
\begin{equation} \label{example.2D.condition.phi}
\lim_{\varepsilon\to0}
\frac{
\varepsilon^{1-\frac{2}{\gamma}}
}
{
\phi(\varepsilon)
}
=0.
\end{equation}
In particular, when $\gamma=1$ the mean value property \eqref{MVP.solid.visco} does not hold point-wise for any function $\phi(\varepsilon)$ satisfying
\eqref{hipotesis.phi}. On the other hand, \eqref{MVP.solid.visco} is satisfied point-wise for every 
$\phi(\varepsilon)$ satisfying
\eqref{hipotesis.phi} when $\gamma=2$.
\end{example}

Using polar coordinates,
\[
\int_{B_{\varepsilon}} u(A x)\,dx
=
\int_{B_{\varepsilon}} |a_{11}x_1+a_{12}x_2|^\gamma\,dx
=
\frac{\varepsilon^{2+\gamma}}{2+\gamma}\int_0^{2\pi} |a_{11}\cos{\theta}+a_{12}\sin{\theta}|^\gamma \,d\theta.
\]
Assume $a_{12}\neq 0 $ and let $\theta_0=\arctan{(a_{11}/a_{12})}$. Then,
\[
\begin{split}
\int_0^{2\pi} &|a_{11}\cos{\theta}+a_{12}\sin{\theta}|^\gamma \,d\theta
=
\left(a_{11}^2+a_{12}^2\right)^{\gamma/2}
\int_0^{2\pi} |\sin{(\theta+\theta_0)}|^\gamma \,d\theta
\\
&
=
\left(a_{11}^2+a_{12}^2\right)^{\gamma/2}
\int_{\theta_0}^{\theta_0+2\pi} |\sin{(\theta+\theta_0)}|^\gamma \,d\theta
=
2\left(a_{11}^2+a_{12}^2\right)^{\gamma/2}
\int_{0}^{\pi} \sin^\gamma{\theta} \,d\theta
\\
&=
C_\gamma\left(a_{11}^2+a_{12}^2\right)^{\gamma/2}.
\end{split}
\]
If $a_{12}= 0 $ we obtain the same result.
Therefore,
\[
\dashint_{B_{\varepsilon}(0) } u(A x)\,dx
=
C_\gamma\left(a_{11}^2+a_{12}^2\right)^{\gamma/2}\varepsilon^\gamma
=
C_\gamma\langle AA^te_1,e_1\rangle^{\gamma/2}\varepsilon^\gamma
\geq
C_\gamma\lambda_{\min}^\gamma(A)\varepsilon^\gamma,
\]
with equality for diagonal matrices where $a_{11}=\lambda_{\min}(A)$.

Observe that the restriction $A\leq \phi(\varepsilon)I$ implies $\lambda_{\min}\geq 1/\phi(\varepsilon)$.
Therefore,
\[
\mathop{\mathop{\inf}_{\det A=1}}_{A\leq \phi(\varepsilon)I } \dashint_{B_{\varepsilon}(0) } u(A x) \,dx
=
C_\gamma
\left(\frac{\varepsilon}{\phi(\varepsilon)}\right)^\gamma,
\]
which is $o(\varepsilon^2)$ only as long as
\eqref{example.2D.condition.phi} holds.
In the particular case  $\phi(\varepsilon)=\varepsilon^{-\alpha}$ 
condition \eqref{example.2D.condition.phi} reduces to $\gamma(1+\alpha)>2$. Then, for each $\alpha\in(0,1)$ there exists $\gamma>1$ such that $\gamma(1+\alpha)\leq 2$ and therefore a function $C^{1,\gamma}$ for which the mean value  property \eqref{MVP.solid.visco}  does not hold point-wise.

If no conditions over the matrix are imposed we get
\[
\inf_{\det A=1} \dashint_{B_{\varepsilon}} u(A x) \,dx =0.
\]	
The restriction $E_{\varepsilon} (A,0) \subset B_1$ implies $\lambda_{\max}\leq 1/\varepsilon$ and therefore $\lambda_{\min}\geq \varepsilon$.
We get
\[
\mathop{\mathop{\inf}_{\det A=1}}_
{E_{\varepsilon} (A,0) \subset B_1}
\dashint_{B_{\varepsilon}} u(A x) \,dx
\geq
C \varepsilon^{2\gamma}
\]
which is not $o(\varepsilon^2)$ for $\gamma=1$.

\begin{example}\label{ejemplopablo}
Here we present a function defined in $\mathbb R^2$ that is a viscosity solution to the equation but that does not verify \eqref{MVP.solid.visco2} in a point-wise sense at the origin.
\end{example}

The function 
\[
u(x_1,x_2)=
\begin{cases}
|x_1| & \text{ if }|x_1|\geq x_2^4\\
\frac{x_2^8+x_1^2}{2x_2^4} & \text{ if }|x_1|< x_2^4.
\end{cases}
\]
It is a solution to  the equation $\det{D^2u(x)}=f(x)$  in $\mathbb{R}^2$ with 
\[
f(x_1,x_2)=
\begin{cases}
0 & \text{ if }|x_1|\geq x_2^4\\
6\frac{x_2^8 - x_1^2}{x_2^{10}} & \text{ if }|x_1|<x_2^4.
\end{cases}
\]
This can be easily verified in the points where the function is smooth. 
Since $u(0,x_2)=\frac{x_2^4}{2}$ and therefore $\frac{\partial^2 u}{\partial^2 x_2}(0,0)=0$ we conclude that the equation holds at the origin.
For the other points of the set $\{|x_1|= x_2^4\}$, we have that $\frac{\partial u}{\partial x_1}(x_1,x_2)=1$ and in the side where $u(x_1,x_2)=|x_1|$ we have $\frac{\partial^2 u}{\partial^2 x_1}(x_1,x_2)=0$.
Then, any convex paraboloid $P$ touching $u$ from below at those points must verify $\frac{\partial^2 P}{\partial^2 x_1}(x_1,x_2)=0$.

Observe that $u(x_1,x_2)\geq |x_1|$ and $u(x_1,x_2)\geq x_2^4/2$, then, we have that
\[
\inf_{\det A=1} \dashint_{B_{\varepsilon} } u(A x) \,d x
\]
is not of order $o(\varepsilon^2)$.

\section{A discrete mean value formula}
\label{sect.another.MVF}

Now our aim is to briefly sketch the proof of the discrete mean value characterization of viscosity solution stated in 
Theorem \ref{eq.mena.value.charact}. To this end we first need to introduce what we understand
for an asymptotic mean value property to hold in the viscosity sense in the discrete case (this definition is similar
to Definition \ref{def.asymp.mean}).

\begin{definition} \label{def.asymp.mean.discrete} A function $u$ verifies the mean value formula
$$
\displaystyle u (x) = \inf_{V\in\mathbb{O}}
\ \inf_{\alpha_{i}\in I_\varepsilon^n} \left\{ \frac{1}{n}
 \sum_{i=1}^n \frac12 u (x + \varepsilon \sqrt{\alpha_i} v_i) + \frac12 u (x - \varepsilon \sqrt{\alpha_i} v_i) 
\right\}  -\frac{\varepsilon^2}{2} (f(x))^{1/n} + o(\varepsilon^2), 
$$
as $\varepsilon \to 0$,
in the viscosity sense if whenever a convex paraboloid $P$ touches $u$ from above (respectively from below) at $x$, it holds that
$$
\displaystyle P (x) \leq (\geq) \inf_{V\in\mathbb{O}}
\ \inf_{\alpha_{i}\in I_\varepsilon^n} \left\{ \frac{1}{n}
 \sum_{i=1}^n \frac12 P (x + \varepsilon \sqrt{\alpha_i} v_i) + \frac12 P (x - \varepsilon \sqrt{\alpha_i} v_i) 
\right\}  -\frac{\varepsilon^2}{2} (f(x))^{1/n} + o(\varepsilon^2), 
$$
\end{definition}

\begin{proof}[Proof of Theorem \ref{eq.mena.value.charact}] Assume that we have a continuous function 
that verifies the asymptotic mean value property. Our goal is to show that it is a viscosity solution. 
To this end let $P$ be a paraboloid that touches $u$ strictly from below at $x$ and with eigenvalues of the Hessian
verifying $\lambda_{i} (D^2 P (x)) > 0$. Then, from the asymptotic mean value property we have
$$
 P (x) \geq \inf_{V\in\mathbb{O}}
\inf_{\alpha_{i}\in I_\varepsilon^n} \left\{ \frac{1}{n}
 \sum_{i=1}^n \frac12 P (x + \varepsilon \sqrt{\alpha_i} v_i) + \frac12 P (x - \varepsilon \sqrt{\alpha_i} v_i) 
\right\}  -\frac{\varepsilon^2}{2} (f(x))^{1/n} + o(\varepsilon^2).
$$
Next, using Taylor expansions as we did in the proof of Theorem \ref{thm.MVP.visco} we get
$$
\displaystyle 0\geq  \inf_{V\in\mathbb{O}}
\inf_{\alpha_{i}\in I_\varepsilon^n}
\left\{\frac{1}{n}\sum_{i=1}^n
 \alpha_{i}  \langle D^2 P (x)v_i, v_i \rangle \right\}
 - (f(x))^{1/n} + o(1).
$$
Passing  to the limit as $\varepsilon \to 0$  we get
$$
\displaystyle 0\geq  \inf_{V\in\mathbb{O}}
\inf_{ \prod_i \alpha_i=1}
\Bigg\{
\frac{1}{n}\sum_{j=1}^n
  \alpha_{i}  \langle D^2 P (x)v_i, v_i \rangle \Bigg\}
 -  (f(x))^{1/n} ,
$$
Since all the $\alpha_{i}$ are positive we conclude 
$$
\det (D^2 P (x)) \leq f (x).
$$

The proof that $u$ is a viscosity subsolution is analogous.

For the converse, assume that $u$ is a viscosity solution and let $ P $ a paraboloid that touches $u$ strictly from below at the point $x$, and with eigenvalues of the Hessian
verifying $\lambda_{i} (D^2 P (x)) > 0$. Since $u$ is  a viscosity solution to $\det (D^2u) =f$, we have that 
$$
\prod_{j=1}^n \lambda_{i} (D^2 P (x)) \leq f (x) .
$$

Using that all the $\alpha_{i}$ are positive we get 
$$
\displaystyle 0\geq  \inf_{V\in\mathbb{O}}
\inf_{ \prod_i \alpha_i=1}\Bigg\{
\frac{1}{n}\sum_{j=1}^n
  \alpha_{i}  \langle D^2 P (x)v_i, v_i \rangle 
  \Bigg\}
 -  (f(x))^{1/n} ,
$$
and hence,  we have 
$$
\displaystyle 0\geq  \lim_{\varepsilon \to 0} 
\inf_{V\in\mathbb{O}}
\inf_{\alpha_{i}\in I_\varepsilon^n}
\Bigg\{
\frac{1}{n}\sum_{i=1}^n
 \alpha_{i}  \langle D^2 P (x)v_i, v_i \rangle 
 \Bigg\} - (f(x))^{1/n} .
$$
That is,
$$
0 \geq  \lim_{\varepsilon \to 0} 
\inf_{V\in\mathbb{O}}
\inf_{\alpha_{i}\in I_\varepsilon^n} \left\{ \frac{1}{n}
 \sum_{i=1}^n 
 \frac{ \displaystyle \frac12 P (x + \varepsilon \sqrt{\alpha_i} v_i) + \frac12 P (x - \varepsilon \sqrt{\alpha_i} v_i) 
 -  P (x)}{\varepsilon^2}
\right\}  -  \frac{1}{2}(f(x))^{1/n} ,
$$
as we wanted to show.

The proof of the other inequality is similar.
\end{proof}
Now, let us go back to Example \ref{example4.4}  to see how this mean value property works using the function
$\phi(\varepsilon)= \frac{1}{\sqrt{\varepsilon}}$. 
We consider the Monge-Amp\`ere equation $\det{D^2u(x)}=f(x)$ in dimension 2 with 
\[ 
u(x) = \frac{1}{2}(|x|-1)_+^2, 
\qquad \mbox{and} \qquad
f(x) = \left(1-\frac{1}{|x|}\right)_+.
\]
As we have mentioned, $u$  is a viscosity solution
of the Monge-Amp\`ere equation with right-hand side $f$, but it is not a classical $C^2$ solution, see Example \ref{example4.4}.

Let us look at the point $x_0=(1,0)$ that is precisely in the set where $u$ is not smooth. 
The mean value property reads as
$$
\displaystyle u (x_0) = 
\inf_{V\in\mathbb{O}}
\inf_{
\alpha_{i}\in I_\varepsilon^2} \left\{ \frac{1}{2}
 \sum_{i=1}^2 \frac12 u (x_0 + \varepsilon \sqrt{\alpha_i} v_i) + \frac12 u (x_0 - \varepsilon \sqrt{\alpha_i} v_i) 
\right\}  -\frac{\varepsilon^2}{2} (f(x_0))^{1/2} + o(\varepsilon^2), 
$$
that is, we would like to see if
$$
\displaystyle 0 = \inf_{V\in\mathbb{O}}
\inf_{
\alpha_{i}\in I_\varepsilon^2} \left\{ \sum_{i=1}^2
  \frac14 (|(1,0) + \varepsilon \sqrt{\alpha_i} v_i)|-1)_+^2 + \frac14  (|(1,0) - \varepsilon \sqrt{\alpha_i} v_i)|-1)_+^2  
\right\}  + o(\varepsilon^2).
$$

As the involved function is non-negative we trivially have
$$
\displaystyle 0 \leq 
\inf_{V\in\mathbb{O}}
\inf_{
\alpha_{i}\in I_\varepsilon^2} \left\{ 
 \sum_{i=1}^2 \frac14 (|(1,0) + \varepsilon \sqrt{\alpha_i} v_i)|-1)_+^2 + \frac14  (|(1,0) - \varepsilon \sqrt{\alpha_i} v_i)|-1)_+^2  
\right\}  + o(\varepsilon^2).
$$

Let us show that 
$$
\displaystyle 0 = 
\frac{1}{\varepsilon^2} 
\inf_{V\in\mathbb{O}} \inf_{
\alpha_{i}\in I_\varepsilon^2} \left\{ 
 \sum_{i=1}^2 \frac14 (|(1,0) + \varepsilon \sqrt{\alpha_i} v_i)|-1)_+^2 + \frac14  (|(1,0) - \varepsilon \sqrt{\alpha_i} v_i)|-1)_+^2  
\right\} , 
$$
To this end, we choose $v_1=(1,0)$ and $v_2 = (0,1)$ and $\alpha_1 = \sqrt{\varepsilon}$ and $\alpha_2 = 1/\sqrt{\varepsilon}$. We obtain
$$
\begin{array}{rl}
\displaystyle 0 & \displaystyle \leq 
\frac{1}{\varepsilon^2} 
\inf_{V\in\mathbb{O}}
\inf_{
\alpha_{i}\in I_\varepsilon^2} \left\{ 
 \sum_{i=1}^2 \frac14 (|(1,0) + \varepsilon \sqrt{\alpha_i} v_i)|-1)_+^2 + \frac14  (|(1,0) - \varepsilon \sqrt{\alpha_i} v_i)|-1)_+^2  
\right\} \\[10pt]
&\leq  \displaystyle  \frac{1}{\varepsilon^2} \left\{ (|(1+ \varepsilon^{5/4},0))|-1)_+^2 +  (|(1, -\varepsilon^{3/4})|-1)_+^2 \right\} \\[10pt]
& =  \displaystyle  \frac{1}{\varepsilon^2} \left\{  \varepsilon^{5/2} +  (\sqrt{1 + \varepsilon^{3/2}}-1)^2 \right\} \\[10pt]
& =  \displaystyle   \varepsilon^{1/2} + \frac{ (\sqrt{1 + \varepsilon^{3/2}}-1)^2}{\varepsilon^2} .
\end{array}
$$
Now, we just have to observe that
$$
\lim_{\varepsilon\to 0} \frac{ (\sqrt{1 + \varepsilon^{3/2}}-1)^2}{\varepsilon^2} 
=  0 
$$
to conclude that the discrete mean value property holds point-wise in this case.

\medskip

%---------------------------------
% BIBLIOGRAFIA
%---------------------------------

\bibliographystyle{plain}

\begin{thebibliography}{99}


\bibitem{Aleksandrov} A. D. Aleksandrov; \emph{Dirichlet's problem for the equation ${\rm Det}\,\vert\vert z\sb{ij}\vert\vert =\varphi (z\sb{1},\cdots,z\sb{n},z, x\sb{1},\cdots, x\sb{n})$. I.} (Russian) Vestnik Leningrad. Univ. Ser. Mat. Meh. Astr. 13 1958 no. 1, 5--24.

\bibitem{Armstrong.2018}
N. Armstrong; \emph{Properties of mean value sets: Angle conditions, blowup solutions,
  and nonconvexity}, Potential Anal, 52 (2020), 527--544.

\bibitem{Armstrong.Blank.2019}
N. Armstrong and I. Blank; \emph{Nondegenerate motion of singular points in obstacle problems with
  varying data}, J. Diff. Eq. (2019).
\bibitem{Angel.Arroyo.Tesis}
A. R. Arroyo Garc\'ia; \emph{Nonlinear Mean Value Properties related to the $p$-Laplacian}, Ph.D. thesis, Universitat Aut\`onoma de Barcelona, (2017).

\bibitem{ArroyoLLorente}  A. Arroyo, and J. G. Llorente, \emph{On the asymptotic mean value property for planar p-harmonic functions}, PAMS 144 (2016), no. 9, pp. 3859--3868.



\bibitem{Aryal.Blank.2019}
A. Aryal and I. Blank; \emph{Geometry of mean value sets for general divergence form uniformly
  elliptic operators}, Potential Anal., 50(1):43--54, 2019.

\bibitem{Bakelman} I. Ya. Bakelman; \emph{On the theory of Monge-Amp\`ere's equations}, (Russian) Vestnik Leningrad. Univ. Ser. Mat. Meh. Astr. 13 1958 no. 1, 25--38.

\bibitem{Benson.Blank.LeCrone.2019}
B. Benson, I. Blank, and J. LeCrone; \emph{Mean value theorems for Riemannian manifolds via the  obstacle problem}, J. Geom. Anal., 29(3):2752--2775, 2019.

\bibitem{[Blanc and Rossi 2019]}  P. Blanc and J. D. Rossi;  \emph{Game Theory and Partial Differential Equations}, Berlin, Boston: De Gruyter (2019).

\bibitem{BH2015}
I. Blank and Z. Hao; \emph{The mean value theorem and basic properties of the obstacle problem  for divergence form elliptic operators},  Comm. Anal. Geom. 23 (2015), no. 1, 129--158.

\bibitem{Blaschke} W. Blaschke; \emph{Ein Mittelwertsatz und eine kennzeichnende Eigenschaft des logarithmischen Potentials}, Ber. Verh. S\"achs. Akad. Wiss., Leipziger, 68 (1916), pp. 3--7.

\bibitem{BoLan} A. Bonfiglioli and E. Lanconelli; {\it Subharmonic functions in sub-Riemannian settings}, 
J. Eur. Math. Soc. (JEMS) 15(2) (2013), 387--441.


\bibitem{Caf90a} L. Caffarelli; \emph{Interior $W\sp {2,p}$ estimates for solutions of the Monge-Amp\`ere equation}, Ann. of Math. (2) 131 (1990), no. 1, 135--150.

\bibitem{Caf90b} L. Caffarelli; \emph{A localization property of viscosity solutions to the Monge-Amp\`ere equation and their strict convexity}, Ann. of Math. (2) 131 (1990), no. 1, 129--134.

\bibitem{Caf91} L. Caffarelli; \emph{Some regularity properties of solutions of Monge Amp\`ere equation} Comm. Pure Appl. Math. 44 (1991), no. 8-9, 965--969.

\bibitem{Caf93} L. Caffarelli; \emph{A note on the degeneracy of convex solutions to Monge Amp\`ere equation} Comm. Partial Differential Equations 18 (1993), no. 7-8, 1213--1217.

\bibitem{Caf98}
L. Caffarelli; \emph{The obstacle problem}, Lezioni Fermiane (Fermi Lectures), Accademia Nazionale dei Lincei,  Rome; Scuola Normale Superiore, Pisa, 1998.


\bibitem{Ca-Ni-Sp} L.A. Caffarelli, L. Nirenberg, and J. Spruck; \emph{The Dirichlet problem for nonlinear second-order elliptic equations. I. Monge-Amp\`ere equation}, Comm. Pure Appl. Math. 37 (1984), no. 3, 369--402.

\bibitem{Ca-Ni-Sp2} L.A. Caffarelli, L. Nirenberg, and J. Spruck; \emph{The Dirichlet problem for nonlinear second-order elliptic equations. III. Functions of the eigenvalues of the Hessian}, Acta Math. 155 (1985), no. 3-4, 261--301.

\bibitem{Caffarelli.Roquejoffre.2007}
L. Caffarelli and J. M. Roquejoffre; \emph{Uniform {H}\"{o}lder estimates in a class of elliptic systems and  applications to singular limits in models for diffusion flames}, Arch. Ration. Mech. Anal., 183(2007), no.3, 457--487.


\bibitem{Guido.Alessio} G. De Philippis, and  A. Figalli; \emph{The Monge-Amp\`ere equation and its link to optimal transportation}, Bulletin of the American Mathematical Society, Vol. 51, No. 4, 2014, p. 527--580.

\bibitem{dTLin} F. Del Teso and E. Lindgreen; {\it A mean value formula for the variational $p-$Laplacian},
Prerprint arXiv:2003.07084v2.

\bibitem{LMan} F. Ferrari, Q. Liu, and J.J. Manfredi; {\it On the characterization of $p-$harmonic functions on the Heisenberg group 
by mean value properties}. Discrete Contin. Dyn. Syst. 34 (2014), no. 7, 2779--2793.


\bibitem{Figalli.2017} A. Figalli; \emph{The Monge-Amp\`ere equation and its applications}, Zurich Lectures in Advanced Mathematics. European Mathematical Society (EMS), Z\"urich, 2017. x+200 pp.


\bibitem{Gaveau}  B. Gaveau; \emph{Methodes de controle optimal en analyse complexe.
I. Resolution d'equations de Monge Amp\`ere.}
Jour. Funct. Anal. 25 (977), 391--411. 

\bibitem{Glimm.Oliker} T. Glimm, and V. Oliker; \emph{Optical design of single reflector systems and the Monge-Kantorovich mass transfer problem}, Nonlinear problems and function theory. J. Math. Sci. (N. Y.) 117 (2003), no. 3, 4096--4108.

\bibitem{Glimm.Oliker2} T. Glimm, and V. Oliker; \emph{Optical design of two-reflector systems, the Monge-Kantorovich mass transfer problem and Fermat's principle}, Indiana Univ. Math. J. 53 (2004), no. 5, 1255--1277.


\bibitem{Gutierrez} C. Guti\'errez; \emph{The Monge-Amp\`ere equation}, Progress in Nonlinear Differential Equations and their Applications, 44. Birkh\"auser Boston, Inc., Boston, MA, 2001. xii+127 pp.


\bibitem{Ku}
\"U. Kuran; \emph{On the mean-value property of harmonic functions},
 Bull. London Math. Soc., 4 (1972), 311--312.


\bibitem{LeGruyerArcher} E. Le Gruyer, and J.C. Archer; \emph{Harmonious Extensions}, Siam J. Math. Anal. Vol 29 (1998), no. 1, pp. 279-292.

\bibitem{LeGruyer} E. Le Gruyer; \emph{On absolutely minimizing extensions and the PDE $\Delta_\infty(u)=0$}, NoDEA 14 (2007), pp. 29--55.

\bibitem{LindqvistManfredi} P. Lindqvist, and J. Manfredi; \emph{On the mean value property for the p-Laplace equation in the plane}, PAMS 144 (2016), no. 1, pp. 143--149.



\bibitem{Lindqvist2014} P. Lindqvist; \emph{Notes on the Infinity Laplace Equation}, Springer (2016).

\bibitem{lions} P. L. Lions; \emph{Two remarks on Monge-Amp\`ere equations}, Ann. Mat. Pura Appl. (4) 142 (1985),   263--275.

\bibitem{lions2} P. L. Lions; \emph{Sur les equations de Monge-Amp\`ere I}, Manuscripta Math. 41, (1983), 1--43. 

\bibitem{Littman.et.al1963}
W. Littman, G. Stampacchia, and H.F. Weinberger; \emph{Regular points for elliptic equations with discontinuous coefficients}, Ann. Scuola. Norm. Sup. Pisa. Cl. Sci., 17(1--2) (1963), 43--77.

\bibitem{[Manfredi et al. 2010]} J. J. Manfredi, M. Parviainen, and J. D. Rossi, \emph{An asymptotic mean value characterization of $p$-harmonic functions}, Proc. Amer. Math. Soc., 138 (2010),  881--889.

\bibitem{PSSW} Y. Peres, O. Schramm, S. Sheffield, and D. Wilson; \emph{Tug-of-war and the infinity Laplacian}, J. Amer. Math. Soc. 22 (2009), no. 1, 167--210.

\bibitem{PSSW2} Y. Peres, and S. Sheffield; {\it Tug-of-war with noise:
a game theoretic view of the $p$-Laplacian}, Duke Math. J. Volume 145, Number 1 (2008), 91-120.

\bibitem{Pogorelov} A. V. Pogorelov; \emph{Extrinsic geometry of convex surfaces}, Translated from the Russian by Israel Program for Scientific Translations. Translations of Mathematical Monographs, Vol. 35. American Mathematical Society, Providence, R.I., 1973. vi+669 pp.

\bibitem{Privaloff} I. Privaloff; \emph{Sur les fonctions harmoniques}, Mat. Sb. 32 (1925), no.3, pp.464--471.



\bibitem{Trudinger3} N. Trudinger; \emph{On the Dirichlet problem for Hessian equations},  Acta Math. 175 (1995), no. 2, 151--164.

\bibitem{Trudinger2} N. Trudinger; \emph{Weak solutions of Hessian equations}, Comm. Partial Differential Equations 22 (1997), no. 7-8,   1251--1261.

\bibitem{Trudinger.Wang} N.S. Trudinger and X.J. Wang; \emph{The Monge-Amp\`ere equation and its geometric applications}, Handbook of geometric analysis 1 (2008): 467-524.

\bibitem{Villani.topics} C. Villani; \emph{Topics in Optimal Transportation}, Graduate Studies in Mathematics, AMS, 2003.


\end{thebibliography}

\end{document}